\documentclass{amsart}

\usepackage{latexsym,amsmath,amssymb,amsfonts,amscd,graphics}
\usepackage[alphabetic,lite,nobysame]{amsrefs}
\usepackage[mathscr]{eucal}
\usepackage{mathrsfs}
\usepackage{tikz-cd}
\usepackage{hyperref}
\usepackage{cleveref}
\usepackage{thmtools}
\usepackage{mathtools}
\usepackage{multicol}
\usepackage[all]{xy}
\usepackage{soul}

\hypersetup{
	colorlinks=true,
	linkcolor=blue,
	filecolor=blue,      
	urlcolor=cyan,
	citecolor=magenta}

\numberwithin{equation}{section}
\theoremstyle{plain}
\newtheorem{theorem}[equation]{Theorem}

\newtheorem{lemma}[equation]{Lemma}
\newtheorem{proposition}[equation]{Proposition}
\newtheorem{corollary}[equation]{Corollary}

\theoremstyle{remark}
\newtheorem{remark}[equation]{Remark}

\theoremstyle{definition}
\newtheorem{definition}[equation]{Definition}

\newcommand{\bP}{\mathbb{P}}

\newcommand{\bR}{\mathbb{R}}

\newcommand{\bZ}{\mathbb{Z}}
\newcommand{\bF}{\mathbb{F}}
\newcommand{\bC}{\mathbb{C}}

\newcommand{\calC}{\mathcal{C}}

\newcommand{\calM}{\mathcal{M}}
\newcommand{\calO}{\mathcal{O}}

\newcommand{\calP}{\mathcal{P}}
\newcommand{\calI}{\mathcal{I}}
\newcommand{\calX}{\mathcal{X}}

\newcommand{\calD}{\mathcal{D}}

\newcommand{\calW}{\mathcal{W}}

\newcommand{\Aut}{\mathrm{Aut}}
\newcommand{\Orth}{\mathrm{O}}
\newcommand{\Bir}{\mathrm{Bir}}
\newcommand{\Bl}{\mathrm{Bl}}

\newcommand{\Span}{\mathrm{span}}
\newcommand{\Hilb}{\mathrm{Hilb}}

\newcommand{\Mov}{\mathrm{Mov}}
\newcommand{\Cone}{\mathrm{Cone}}

\newcommand{\im}{\mathrm{Im}}
\newcommand{\id}{\mathrm{Id}}

\newcommand{\Gr}{\mathrm{Gr}}

\newcommand{\pex}{\mathrm{pex}}
\newcommand{\flop}{\mathrm{flop}}

\newcommand{\rank}{\mathrm{rank}}

\newcommand{\Nef}{\mathrm{Nef}}

\newcommand{\NS}{\mathrm{NS}}

\newcommand{\Bis}{\mathrm{Bis}}
\newcommand{\Mon}{\mathrm{Mon}}

\newcommand{\git}{/\kern-0.2em/}

\newcommand{\prim}{\mathrm{prim}}
\newcommand{\Amp}{\mathrm{Amp}}

\title[Lines on cubic fourfolds containing pairs of cubic scrolls]{Birational geometry of Fano varieties of lines on cubic fourfolds containing pairs of cubic scrolls}

\author{Corey Brooke}
\address{Department of Mathematics and Statistics, Carleton College, 1 North College St, Northfield, MN, 55057}
\email{cbrooke@carleton.edu}

\author{Sarah Frei}
\address{Department of Mathematics MS 136, Rice University, 6100 S. Main St, Houston, TX 77005, USA}
\email{sarah.frei@rice.edu}

\author{Lisa Marquand}
\address{Courant Institute,
  251 Mercer Street,
  New York, NY 10012, USA}
\email{lisa.marquand@nyu.edu}

\author{Xuqiang Qin}
\address{Department of Mathematics, 329 Phillips Hall, Chapel Hill, NC 27599}
\email{russellqin@gmail.com}

\begin{document}

\maketitle

\begin{abstract}
    We characterize the birational geometry of some hyperk\"ahler fourfolds of Picard rank $3$ obtained as the Fano varieties of lines on cubic fourfolds containing pairs of cubic scrolls. In each of the two cases considered, we identify all of the birational models, relating each model to familiar geometric constructions, and give explicit birational maps between them. We also provide structural results about the birational automorphism groups, giving generators in both cases and a full set of relations in one case. Finally, as a byproduct of our analysis, we obtain non-isomorphic cubic fourfolds whose Fano varieties of lines are birationally equivalent.
\end{abstract}

\section{Introduction}

Cubic fourfolds are a central object in algebraic geometry, studied with respect to rationality questions and for their connections to hyperk\"ahler manifolds. In a very general cubic fourfold, any algebraic surface is homologous to a complete intersection, but a countably infinite collection of divisors $\calC_d$ in the moduli space of cubic fourfolds parametrize cubics containing extra algebraic surfaces \cite{hassett}. Cubics contained in these divisors are often more interesting from the point of view of rationality questions: it is conjectured that the rational cubic fourfolds are contained in the union of certain divisors $\calC_d$.  The Fano variety of lines $F$ on a cubic fourfold $X$ is a hyperk\"ahler manifold by \cite{MR818549}, and when $X\in \calC_d$, typically $F$ exhibits richer birational geometry, which can be studied via the Global Torelli theorem (due to Huybrechts, Markman and Verbitsky). Here, we focus on the Fano varieties of lines on cubic fourfolds belonging to the family $\calC_{12}$, whose general member contains a cubic scroll. 

The divisor $\calC_{12}$ was first studied in \cite{flops}, where the authors studied the birational models of the corresponding Fano variety of lines, and exhibited a birational transformation of infinite order. More recently, it has been proved that $\calC_{12}$ contains many geometrically interesting families of cubic fourfolds. For example, there is a ten dimensional family of cubics admitting an involution fixing a line pointwise, denoted by $\calM_{\phi_2}$ in the notation of \cite{marquand2022cubic}. For $X\in \calM_{\phi_2},$ the middle algebraic cohomology of $X$ is spanned by classes represented by cubic scrolls, along with the square of the hyperplane class. Any pair of cubic scrolls spanning different hyperplane sections of $X$ is either \textbf{syzygetic} or \textbf{non-syzygetic}, meaning they intersect in three or one points, respectively (\Cref{defn: syz}). Further, such a cubic fourfold is conjecturally irrational \cite[Theorem 1.2]{marquand2022cubic}, and one could hope this is reflected in the birational geometry of the associated hyperk\"ahler manifolds. 

Motivated by this, we study the birational geometry of  the Fano variety $F$ of lines on a very general cubic fourfold $X$ containing either a syzygetic or a non-syzygetic pair of cubic scrolls. In both cases, $F$ has Picard rank three, and we describe the movable and ample cones of $F$ by identifying the wall divisors following techniques from \cite{MR3423735}.
We first consider a cubic fourfold $X$ with a syzygetic pair of cubic scrolls --- in this case, the movable cone is bounded by infinitely many walls corresponding to prime exceptional divisors (see \Cref{subsec: HK}). We prove the following result:
    \begin{theorem}\label{thm:syzygeticintro}
    Let $F$ be the Fano variety of lines on a very general cubic fourfold $X$ containing a syzygetic pair of cubic scrolls. Then $F$ has \textbf{five} isomorphism classes of birational hyperk\"ahler models, represented by the following:
    \begin{itemize}
        \item $F$ itself, and
        \item four non-isomorphic Mukai flops of $F$,  each isomorphic to a double EPW sextic.
    \end{itemize}
    \end{theorem}

In the non-syzygetic case, the movable cone coincides with the positive cone, but surprisingly the birational geometry of $F(X)$ is more complicated. This is partially explained by the fact that cubics containing a non-syzygetic pair of cubic scrolls automatically contain a third family of cubic scrolls (see \Cref{lemma:extrascrolls}).

    \begin{theorem}\label{thm:nonsyzygeticintro}
    Let $F$ be the Fano variety of lines on a very general cubic fourfold $X$ containing a non-syzygetic pair of cubic scrolls. Then $F$ has \textbf{eight} isomorphism classes of birational hyperk\"ahler models, represented by the following: 
    \begin{itemize}
        \item $F$ itself,
        \item six non-isomorphic Mukai flops of $F$, each isomorphic to a double EPW sextic, and
        \item a Mukai flop of $F$ along a pair of disjoint planes in $F$, isomorphic to the Fano variety of lines on another cubic fourfold containing a non-syzygetic pair of cubic scrolls.
    \end{itemize}
    \end{theorem}

As a consequence, we obtain the first examples of the following phenomenon:

    \begin{corollary}
        There exist non-isomorphic cubic fourfolds with birationally equivalent Fano varieties of lines.
    \end{corollary}

The relationship between these cubics is explored more fully in two later papers. In \cite{BFM}, they are shown to be birationally equivalent Fourier--Mukai partners, and the authors of \cite{BGM} use Gale duality to relate the equations of the two cubics.

These results are of interest for three main reasons. First, they provide examples of hyperk\"ahler fourfolds of Picard rank three whose birational geometry is explicitly understood; to our knowledge, this has only been done in the significantly easier case of Picard rank two.
Second, our techniques also allow us to deduce information about the birational automorphism group of $F$, even though understanding the structure of $\Bir(F)$ is in general more difficult than enumerating the birational models. Specifically, we obtain the following structural result for $\Bir(F)$:

\begin{theorem}\label{theo:bir}
    Let $F$ be the Fano variety of lines on $X$, a very general cubic fourfold containing a syzygetic or non-syzygetic pair of cubic scrolls. Then:
    \begin{enumerate}
        \item The infinite order group $\Bir(F)$ is generated by the covering involutions of the double EPW sextics obtained as Mukai flops of $F.$
        \item In particular, in the syzygetic case we have:
        \[
    \Bir(F)\cong\langle a,b,c,d\;|\; a^2=b^2=c^2=d^2=1\rangle.
    \]
    \end{enumerate}
\end{theorem}

Finally, the birational equivalence between $F$ and double EPW sextics continues a beautiful story of associations between cubics in $\calC_{12}$ and Gushel--Mukai fourfolds. 
In \cite[Section 7.2]{DIM}, the authors show that a general $X\in \calC_{12}$ is birational to a Hodge-theoretically associated to a GM fourfold $Z$ containing a plane; in \cite[Theorem 5.8]{KuzPerry}, $X$ and $Z$ are moreover shown to be Fourier--Mukai partners. The fourfold $Z$ in turn determines a double EPW sextic, a hyperk\"ahler manifold introduced by O'Grady (\cite{OG1, OG2, OG3}). 
In \cite{IM11}, the authors show these manifolds can be constructed from considering conics on the general Gushel--Mukai fourfold. This result was extended to all smooth ordinary Gushel--Mukai fourfolds in \cite[Theorem 7.12]{DK24}. Combining both the birational map between $X$ and $Z$ and this geometric description of the double EPW sextic in terms of conics on $Z$, we construct birational maps between $F$ and two non-isomorphic double EPW sextics explicitly. Along with correcting the census of birational models of $F$ given in \cite{flops}, this yields a clearer geometric interpretation of each model and of the birational automorphism group, persisting even when $X$ specializes to both the syzygetic and non-syzygetic cases.

\subsection*{Outline} In Section~\ref{sec:prelims}, we recall the relevant definitions and results about cubic threefolds and fourfolds containing cubic scrolls, the Fano variety of lines on such a cubic fourfold, and the birational geometry of hyperk\"ahler manifolds of K3$^{[2]}$-type. We also introduce Gushel--Mukai fourfolds and double EPW sextics.
In Section~\ref{sec:onescroll}, we study the birational geometry of the Fano variety of lines $F$ on a very general cubic fourfold $X$ containing a cubic scroll, proving that $F$ is birational to two non-isomorphic double EPW sextics. In particular, Theorem~\ref{mainthm:1scroll} corrects the census of birational models of $F$ originally provided in \cite[Theorem 7.4]{flops}. In Sections~\ref{sec:syzygetic} and~\ref{sec:nonsyzygetic}, we prove Theorems~\ref{thm:syzygeticintro} and~\ref{thm:nonsyzygeticintro} (Theorems~\ref{thm:syzmain} and~\ref{theo:nonsyzygeticmain}), respectively, as well as Theorem~\ref{theo:bir} (Theorems~\ref{theo:birFsyz} and~\ref{theo:nonsyzygeticbir}), by carrying out an analysis of the chambers of the movable cone of $F$ when $X$ contains a syzygetic and non-syzygetic, respectively, pair of cubic scrolls. Explicit examples in Appendix~\ref{appendix:examples} illustrate generic behavior of such cubic fourfolds.

\subsection*{Acknowledgments} We thank Nicolas Addington, Brendan Hassett, Alex Perry, Jack Petok, and Yuri Tschinkel for valuable conversations. In particular, Nicolas Addington suggested the possible connection to double EPW sextics explored in Section~\ref{sec:epw}, and Brendan Hassett suggested techniques that streamlined our arguments enumerating the birational models of $F$.
The computations in Appendix~\ref{appendix:examples} were done in Magma \cite{Magma}.
Finally, we thank the anonymous referees and Alexander Kuznetsov for their comments and remarks that helped the exposition.

This material is based upon work supported by the NSF Grant DMS-1929284 while the authors were in residence at the Institute for Computational and Experimental Research in Mathematics in Providence, RI, during the Hyperk\"ahler Manifolds and Special Cubic Fourfolds Collaborate@ICERM Program. The authors were also partially supported by the Hausdorff Research Institute for Mathematics, while in residence for the Junior Trimester Program on Algebraic geometry: derived categories, Hodge theory, and Chow groups, funded by the Deutsche Forschungsgemeinschaft (DFG, German Research Foundation) under Germany’s Excellence Strategy – EXC-2047/1 – 390685813. S.F.~was supported in part by an AMS-Simons travel grant and NSF grant DMS-2401601. L.M. was supported in part by NSF grant DMS-2503390.

\section{Preliminaries}\label{sec:prelims}
We consider complex cubic threefolds and fourfolds that contain rational normal cubic scrolls, hereafter referred to as \emph{cubic scrolls}. 
\begin{definition}
    A \emph{cubic scroll} $T\subset \bP^4$ is an irreducible non-degenerate surface of degree 3. Equivalently, $T$ is isomorphic to the blow up of $\bP^2$ in a point embedded into $\bP^4$.
\end{definition}

In \Cref{subsection:threefolds} we recall generalities on nodal cubic threefolds $Y\subset \bP^4$ that contain cubic scrolls. In particular, we recall the description of the components of the Fano variety of lines $F(Y)$ for such a threefold due to \cite{flops}. In \Cref{subsec: special cubic}, we focus on smooth cubic fourfolds $X\subset \bP^5$ containing cubic scrolls, defining the notion of syzygyetic and non-syzygetic pairs of scrolls (\Cref{defn: syz}). In \Cref{subsec: Fano prelims}, we describe the Abel--Jacobi map relating the cohomology of $X$ to its Fano variety of lines. In \Cref{subsec: HK}, we describe the structure of the movable and ample cone of a hyperk\"ahler manifold of K3$^{[2]}$-type (\Cref{thm: Mov and Amp}). Finally, in \Cref{subsec: GM4folds} we recall the relevant results on Gushel--Mukai fourfolds and associated double EPW sextics, emphasizing a geometric point of view.

\subsection{Cubic threefolds containing cubic scrolls}\label{subsection:threefolds}

Let $Y$ be a cubic threefold containing a cubic scroll $T$, as studied in detail in \cite{flops}. A generic such cubic threefold has six nodes in general position. Conversely, any cubic threefold with six nodes in general position contains a cubic scroll. Indeed, such a cubic threefold is determinental, and the cubic scrolls in $Y$ correspond to the degeneracy loci of 2 by 3 and 3 by 2 minors of this determinental matrix. We summarize some of their geometric properties in this section.

By \cite[Proposition 4.7]{flops}, $Y$ contains another cubic scroll $T^\vee$ residual to $T$ in a quadric, and two nets of cubic scrolls homologous to $T$ and $T^\vee$ respectively. 
Each node of $Y$ is contained in all of the cubic scrolls in $Y$, and the complete linear system $|\calI_{T/Y}(2)|$ induces a rational map $q\colon Y\dashrightarrow\bP^2$, resolved by passing to a small resolution $Y^+$ of $Y$ as in the diagram below.
\begin{center}
        \begin{equation}\label{diag:smallresolution}
        \begin{tikzcd}
            & Y^+ \arrow[dr,"p"] \arrow[dl,"f" above] & \\ 
            Y \arrow[rr,dashed,"q"] & & \bP^2
        \end{tikzcd}
        \end{equation}
\end{center} 
As noted in \cite[Remark 7.1]{cheltsov2024equivariant}, the map $p$ is a $\bP^1$-bundle, and for a general line $\ell\subset \bP^2$, $f_*p^*(\ell)$ is a cubic scroll homologous to $T^\vee$. 

\begin{remark}\label{remark:lines in T}
    If $T'\subset Y$ is a smooth cubic scroll homologous to $T^\vee$, there are two types of lines in $T'$: an exceptional line $E$ with $E^2=-1$ under the intersection form and a pencil of lines $L$ with $L^2=0$. The first type is a section of the $\bP^1$-bundle $q|_{T'}\colon T'\to\bP^1$ whereas the second is a fiber.
\end{remark}

Similarly, the linear system $|\calI_{T^\vee/Y}(2)|$ induces a map $q^\vee\colon Y\dashrightarrow\bP^2$ resolved by a small resolution $Y^-$ which differs from $Y^+$ by six Atiyah flops \cite[Section~7]{cheltsov2024equivariant}. 

The Fano variety of lines on $Y$ has three components:

\begin{proposition}\label{prop: P, S, Pv}\cite[Proposition 4.1]{flops}
    The Fano variety of lines $F(Y)$ on $Y$ decomposes as $P\cup S'\cup P^\vee$ where $P,P^\vee\simeq\bP^2$ and $S'$ is a singular surface whose normalization is a smooth cubic surface.
\end{proposition}

More specifically, the normalization $S\to S'$ resolves fifteen apparent double points on $S'$, corresponding to the fifteen lines passing through two nodes of $Y$ \cite[Proposition 4.6]{flops}.

\begin{remark}\label{remark:fanodescription}
We summarize another account of the components of $F(Y)$, as outlined in \cite{dolgnodal}. Given a node $y\in Y$, a union of twisted cubic curves $C_y$ and $C_y^\vee$ parametrizes the lines on $Y$ passing through $y$ \cite[Lemma 4.4]{flops}. Hence the surface in $Y$ swept out by the lines through $y$ decomposes as a union of cones $A_y\cup A_y^\vee$ over twisted cubics; by \cite[Proposition 4.7]{flops}, the surfaces $A_y$ and $A_y^\vee$ are homologous to $T$ and $T^\vee$, respectively. One obtains a birational map
\[
\Hilb^2(C_y\cup C_y^\vee)\dashrightarrow F(Y)
\] by taking the residual line to the two lines determined by $\xi\in\Hilb^2(C_y\cup C_y^\vee).$
Considering instead bisecant lines, which intersect $C_y\cup C_y^\vee$ with multiplicity at least $2$, Dolgachev upgrades this to an isomorphism
\[
\Bis(C_y\cup C_y^\vee)\xlongrightarrow{\sim} F(Y)
\]
in \cite[Section 4]{dolgnodal}. The planes $P,P^\vee\subset F(Y)$ can be identified with $\Bis(C_y)$ and $\Bis(C_y^\vee)$, i.e. the set of lines bisecant to $C_y$ and $C_y^\vee$, respectively. The surface $S'$ consists of lines meeting both $C_y$ and $C_y^\vee$. 
\end{remark}

\subsection{Special cubic fourfolds in $\calC_{12}$}\label{subsec: special cubic}

For a smooth cubic fourfold $X$, we define the lattice of algebraic cycles (using the fact that cubic fourfolds satisfy the integral Hodge conjecture \cite[Theorem 1.4]{voisinHodge}):
\[
A(X)=H^4(X,\bZ)\cap H^{2,2}(X,\bC).
\]
Letting $\eta_X$ be the square of the hyperplane class, we also define
\[
H^4(X,\bZ)_{prim}=\langle\eta_X\rangle^\perp\subset H^4(X,\bZ)
\]
and
\[
A(X)_{prim}=H^4(X,\bZ)_{prim}\cap A(X).
\]
Each of the above is a lattice under the intersection pairing. For a very general cubic fourfold, $A(X)_{prim}=0$, but countably many divisors $\calC_d$ in the moduli space of cubic fourfolds parametrize those with $A(X)_{prim}\neq0$. Whenever $\eta_X\in K\subset A(X)$ where $K$ is a primitive sublattice of rank two, and the discriminant of the intersection form on $K$ is $d$, we say $X\in\calC_d$, and $K$ is called a labeling of $X$.

We focus in particular on smooth cubic fourfolds $X\subset \bP^5$ that contain a cubic scroll $T$, or equivalently contain a hyperplane section $Y=X\cap H$ with six nodes in general position \cite[Prop. 6.1]{flops}. A cubic scroll $T\subset X$ determines a labelling $K_{12}\hookrightarrow A(X)$ (see \cite[Section 4.13]{hassett}), so $X\in \calC_{12}.$ Indeed, the intersection pairing on $\langle\eta_X,T\rangle\subset A(X)$ is given by:
\begin{center}
\begin{tabular}{l|ll}
& $\eta_X$ & $T$  \\ \hline
$\eta_X$ & $3$   & $3$   \\
$T$ & $3$   & $7$  \\
\end{tabular}
\end{center}
As a consequence of Section~\ref{subsection:threefolds}, we immediately see the following:

\begin{lemma}\cite[Lemma 2.11]{hassett-thesis}
    Let $X$ be a general cubic fourfold containing a cubic scroll $T$. Then:
    \begin{enumerate}
        \item there is a net of scrolls homologous to $T$, all contained in $H$,
        \item there exists a residual cubic scroll $T^\vee\subset X$ spanning the same hyperplane as $T$ such that $[T]+[T^\vee]=2\eta_X\in A(X).$ 
    \end{enumerate}
\end{lemma}

In much of what follows, we take $X$ to contain two cubic scrolls $T_1$ and $T_2$ spanning different hyperplanes $H_1$ and $H_2$. For a very general such $X$, the classes $\eta_X$, $[T_1]$, and $[T_2]$ span $A(X)$, and the intersection form is given by
\begin{center}
$M_\tau:=$
\begin{tabular}{l|lll}
& $\eta_X$ & $T_1$ & $T_2$  \\ \hline
$\eta_X$ & $3$   & $3$ &  $3$ \\
$T_1$ & $3$   & $7$ & $\tau$ \\
$T_2$ & $3$ & $\tau$ & $7$ \\
\end{tabular}
\end{center}
where $\tau\in\{1,2,3,4,5\}$ \cite[Lemma 4.11]{marquand2022cubic}. Note that if $[T_1]\cdot [T_2]=\tau,$ then $[T_1^\vee]\cdot [T_2]=6-\tau$ since $[T_1]+[T_1^\vee]=2\eta_X$; hence we can take $\tau\in\{1,2,3\}$.

\begin{remark}
    Note that there is a small error in \cite[Lemma 4.11]{marquand2022cubic}, which originally ruled out $\tau=2,4$: such an intersection gives a vector $v= 2\eta_X-[T_1]-[T_2]$ with $v^2=6$, but $\mathrm{div} (v)=1$ in the full primitive cohomology $H^4(X, \bZ)$.
\end{remark} 

\begin{lemma}\label{lem:Sigmasmooth}
    Let $X\subset \bP^5$ be a general cubic fourfold containing two cubic scrolls $T_1, T_2$ spanning distinct hyperplanes $H_1, H_2\subset \bP^5$. Then the cubic surface $\Sigma \coloneqq H_1\cap H_2\cap X$ is smooth.
\end{lemma}
\begin{proof}
    We provide examples of such cubics in \Cref{appendix:examples} when $\tau=1,3$. Since smoothness is an open condition, and the moduli spaces of cubics with markings by $M_{\tau}$ for $\tau = 1, 3$ are both irreducible by \cite[Theorem~1.1]{yang2021lattice}, the result follows in those cases. A similar example proves the claim for $\tau=2$ but is not included because we do not treat such cubics in this paper.
\end{proof}
 Thus, one can interpret the intersection of $T_1$ and $T_2$ in terms of the intersection of two twisted cubics on $\Sigma=X\cap H_1\cap H_2$, (see \cite[Appendix B.1]{marquand2022cubic}). This observation motivates the following definition:
 \begin{definition}\label{defn: syz}
     We say that $T_1$ and $T_2$ are \emph{syzygetic scrolls} provided $[T_1]\cdot [T_2]=3,$ and \emph{azygetic scrolls} provided $[T_1]\cdot [T_2]=2$ or $4$.
     If $[T_1]\cdot [T_2]=1$ or $5$ we say the pair is \emph{non-syzygetic}.
 \end{definition}
Indeed, the twisted cubics $T_1\cap H_2$ and $T_2\cap H_1$ on the cubic surface $\Sigma=X\cap H_1\cap H_2$ form a syzygetic duad in the first case, as defined in \cite[Lemma 9.1.5]{dolg}, and an azygetic duad in the second. We study these twisted cubics in Lemma~\ref{lemma:gammaintersection}.

In the non-syzygetic case, there are exactly two other algebraic cycles with the same numerics as a cubic scroll, namely:
\begin{align}\label{eqn:T3def}
\begin{split}
    [T_3]&=3\eta_X-[T_1]-[T_2],\\
    [T_3^\vee]&=[T_1]+[T_2]-\eta_X.
\end{split}
\end{align} 
\begin{lemma}\label{lemma:extrascrolls}
    The classes $[T_3]$ and $[T_3^\vee]$ are represented by cubic scrolls in $X$.
\end{lemma}
\begin{proof}
Let $\calX\rightarrow B$ be a local universal family of marked cubic fourfolds, where we identify each lattice $A(X_b)_{\prim}$ with a sublattice of $H^4(X_0,\bZ)_{\prim}$, with $X_0=X$. We let $B'\subset B$ denote the Hodge locus of the class $[T_3]$; i.e. the locus parametrizing fibers $\calX_b$ where $[T_3]$ remains algebraic, so $[T_3]\in A(X_b)$. If $b\in B'$ is very general, then the cubic $X_b$ has $A(X_b)\cong \langle \eta_{X_b}, [T_3]\rangle$ and $X_b$ contains a cubic scroll $T'$, necessarily with class $[T_3]$ \cite[Section 4.1.2]{hassett-thesis}. Then by specialization, in $X$ the class $[T_3]$ is also represented by an effective cycle of degree 3 in $\bP^5.$ Since $X$ does not contain a plane, this is necessarily irreducible, and since $[T_3]^2=7$ the only option is for $[T_3]$ to be represented by a cubic scroll. 
Since $[T_3^\vee]=2\eta_X-[T_3]$, we know $[T_3^\vee]$ is also represented by a cubic scroll.
\end{proof}

We will also make use of the following fact:

\begin{proposition}\label{prop:Aut(X)=0}
    For $X$ a general cubic fourfold in $\calC_{12}$ or a general cubic fourfold containing a pair of cubic scrolls, $\Aut(X)=0$.
\end{proposition}
\begin{proof}
    The locus of cubic fourfolds with an order $p$ automorphism has dimension at most dimension 14 (see \cite[Theorem 3.8]{Gonz}). It follows that a general $X$ is not contained in any of these loci.
\end{proof}

\begin{remark}
    Our main motivation for studying the cubics containing syzygetic and non-syzygetic pairs of scrolls stems from the desire to investigate the rationality of cubics $X$ with an involution fixing a line pointwise. The twelve dimensional moduli space of such cubics $\calM_{\phi_2}$ is contained in $\calC_{12}$: however, a general $X\in \calM_{\phi_2}$ contains 240 classes represented by cubic scrolls, and $\rank A(X)_{prim}=8$ \cite[Theorem 4.5]{marquand2022cubic}. Thus the results discussed here can be viewed as a step towards understanding cubics contained in this family, to be treated in future work. 
\end{remark}

\subsection{The Fano variety of lines on a cubic fourfold}\label{subsec: Fano prelims}
Let $X$ be a smooth cubic fourfold and $F$ its Fano variety of lines,  a hyperk\"ahler fourfold of K3$^{[2]}$-type \cite{MR818549}. 

The Abel--Jacobi map
\[
\alpha\colon H^4(X,\bZ)\rightarrow H^2(F,\bZ)
\]
restricts to an isomorphism $H^4(X,\bZ)_{prim}\rightarrow H^2(F,\bZ)_{prim}$ compatible with the Hodge filtrations and (up to sign) with the quadratic forms on each lattice; more specifically, we have
\[
q(\alpha(x),\alpha(y))=-x\cdot y
\]
where $q$ is the Beauville--Bogomolov--Fujiki form (BBF form) of $F$. We write $v^2\coloneqq q(v,v)$ for $v \in H^2(F,\bZ)$. Letting $g$ be the Pl\"ucker polarization on $F$, we have that $\alpha(\eta_X)=g$ and, moreover, $g^2=6$ \cite[Proposition 6]{MR818549}. 

If $A(X)=\langle \eta_X,T\rangle$ where $T$ is a cubic scroll, then setting $\lambda=\alpha(T-\eta_X)$ and applying the compatibilities of the Abel--Jacobi map as described above, we see that the BBF form on $\NS(F)$ is given by
\begin{center}
$J_{12}=$
\begin{tabular}{r|rr}
& $g$ & $\lambda$  \\ \hline
$g$ & $6$   & $0$   \\
$\lambda$ & $0$   & $-4$  \\
\end{tabular}
\end{center}
and the discriminant group of $\NS(F)$ is
\[
D_{\NS(F)}\cong\bZ/6\times\bZ/4,
\]
with factors generated by $[\frac g6]$ and $[\frac{\lambda}4]\in D_{\NS(F)}$.

Next, suppose $A(X)=\langle \eta_X,T_1,T_2\rangle$ where $T_1$ and $T_2$ are cubic scrolls spanning different hyperplanes. Let $\lambda_i=\alpha(T_i-\eta_X)$. If $T_1$ and $T_2$ are a syzygetic pair, then the BBF form on $\NS(F)$ is 
\begin{center}
$J_{syz}=$
\begin{tabular}{r|rrr}
& $g$ & $\lambda_1$ & $\lambda_2$ \\ \hline
$g$ & $6$   & $0$ & $0$   \\
$\lambda_1$ & $0$   & $-4$ & $0$ \\
$\lambda_2$ & $0$ & $0$ & $-4$\\
\end{tabular}
\end{center}
and the discriminant group
\[
D_{\NS(F)}\cong\bZ/6\times(\bZ/4)^2
\]
has factors generated by $[\frac g6]$, $[\frac{\lambda_1}4]$, and $[\frac{\lambda_2}4]$. On the other hand, if $T_1$ and $T_2$ are a non-syzygetic pair labeled so that $[T_1]\cdot[T_2]=1$, the BBF form on $\NS(F)$ is
\begin{center}
$J_{nonsyz}=$
\begin{tabular}{r|rrr}
& $g$ & $\lambda_1$ & $\lambda_2$ \\ \hline
$g$ & $6$   & $0$ & $0$   \\
$\lambda_1$ & $0$   & $-4$ & $2$ \\
$\lambda_2$ & $0$ & $2$ & $-4$\\
\end{tabular}
\end{center} 
and the discriminant group 
\[
D_{\NS(F)}\cong\bZ/6\times\bZ/2\times\bZ/6
\]
has factors generated by $[\frac g6]$, $[\frac{\lambda_1}2]$, and $[\frac{\lambda_1}3+\frac{\lambda_2}6]$.

\subsection{Birational geometry of hyperk\"ahler fourfolds}\label{subsec: HK}
In order to study the birational geometry of $F$, we use the Global Torelli theorem for hyperk\"ahler manifolds of K3$^{[2]}$-type due to Verbitsky \cite{VerbTorelli}. It was reformulated by Huybrechts \cite{HuyTorelli} and Markman \cite{markman}; in particular, we use Markman's Hodge-theoretic version. We denote by $\Mon^2_{Hdg}(F)$ the subgroup of monodromy operators in $\Mon^2(F)$ that preserve the Hodge structure. Recall that for a manifold $F$ of K3$^{[2]}$-type, $\Mon^2(F)$ is equal to the subgroup of elements of $\Orth^+(H^2(F,\bZ))$ that act by $\pm \id$ on the discriminant group \cite[Lemma 9.2]{markman}.
\begin{theorem}\cite[Theorem 1.3]{markman}\label{thm:torelli}
    Let $F$ be a projective hyperk\"ahler manifold.
    Let $g\in \Mon^2_{Hdg}(F)$. Then there exists $f\in \Bir(F)$ such that $f^*=g$ if and only if $g^*\Mov(F)=\Mov(F)$. Further, $f\in \Aut(F)$ if and only if $g^*\Amp(F)=\Amp(F)$.
\end{theorem}
\begin{remark}\label{remark: when bir is reg}
    In particular, it follows that a birational map $f\colon F\dashrightarrow F'$ between hyperk\"ahler manifolds is regular (hence an isomorphism) if and only if $f^*\omega$ is ample for an ample class $\omega$.
\end{remark}
The structure of the nef and movable cones of hyperk\"ahler manifolds are well known: for K3$^{[m]}$-type, they are described in \cite[\S 9]{markman} and \cite[\S 1]{BHT} using the extended Mukai-lattice. Here, we focus on the simpler case of manifolds of K3$^{[2]}$-type.

Let $F$ be a projective hyperk\"ahler manifold of K3$^{[2]}$-type. Let $\overline{\mathrm{Pos}(F)}$ be the component of the cone $\{x\in \NS(F)\otimes \bR \mid x^2 \geq 0\}$ that contains an ample class. We denote by $\Mov(F)\subset \overline{\mathrm{Pos}(F)}$ the (closed) convex cone generated by classes of line bundles on $F$ whose base locus has codimension at least $2$. By \cite[Theorem 7]{movingcones},
$$\Mov(F) = \overline{\bigcup_{f\colon F\dashrightarrow F'} f^*\Nef(F')},$$ where $f\colon F\dashrightarrow F'$ is a birational map with $F'$ a hyperk\"ahler manifold.

\begin{remark}
    For hyperk\"ahler manifolds with $b_2\geq 5$, there are finitely many birational models (see \cite{MR3436156}, \cite{MR3679618}, and  \cite{AV20} for the removal of $b_2\neq5$ assumption). In particular, this is the case for hyperk\"ahler manifolds of K3$^{[2]}$-type.
\end{remark}

We define the following set of divisors:
\begin{align*}
    \calW_{\pex}&\coloneqq \{\rho \in \NS(F) \mid \rho^2=-2 \}\\
    \calW_{\flop}&\coloneqq \{\rho\in \NS(F) \mid \rho^2=-10, \textrm{div}(\rho)=2\}.
\end{align*}
A divisor $\rho\in \calW_{\pex}\cup \calW_{\flop}$ is called a wall divisor \cite[Definition 1.2, Proposition 2.12]{MR3423735}, and $\rho\in \calW_{\pex}$ is moreover referred to as a stably prime exceptional divisor. The following structure theorem for $\Mov(F)$ and $\Amp(F)$ in the case of hyperk\"ahler manifolds of K3$^{[2]}$-type combines the results of Markman \cite[Lemma 6.22, Proposition 6.10, 9.12, Theorem 9.17]{markman} and  Bayer, Hassett and Tschinkel \cite[Theorem 1]{BHT}.

\begin{theorem}\cite[Theorem 3.16]{debarre2020hyperkahler}\label{thm: Mov and Amp}
    Let $F$ be a hyperk\"ahler manifold of K3$^{[2]}$-type.

    \begin{enumerate}
        \item The interior of $\Mov(F)$ is the connected component of $$\overline{\mathrm{Pos}(F)}\setminus \bigcup_{\rho\in \calW_{\pex}} \rho^\perp$$ that contains the class of an ample divisor.
        \item The ample cone $\Amp(F)$ is the connected component of 
        $$\overline{\mathrm{Pos}(F)}\setminus \bigcup_{\rho\in \calW_{\pex}\cup \calW_{\flop}} \rho^\perp$$ that contains the class of an ample divisor.
    \end{enumerate}
\end{theorem}

Note that each connected component of 
\[
\textrm{Int}(\Mov(F))\setminus \bigcup_{\rho\in \calW_{\flop}} \rho^\perp
\]
corresponds to $f^*(\Amp(F'))$ for a birational map $f\colon F\dashrightarrow F'$ with $F'$ a hyperk\"ahler manifold.

To obtain this chamber decomposition of $\overline{\mathrm{Pos}(F)}$ and of $\Mov(F),$ Markman uses a more explicit description of the group $\Mon^2_{Hdg}(F)$ as follows.
Let $W_{Exc}\subset \Mon^2_{Hdg}(F)$ be the subgroup generated by reflections $R_\rho$ for all $\rho\in \calW_{\pex}$, and let $\Mon_{Bir}^2\subset \Mon^2_{Hdg}(F)$ be the subgroup generated by monodromy operators induced from birational transformations of $F$.
Let $\pi\colon \Mon_{Hdg}^2(F)\rightarrow \Orth(\NS(F))$ be the restriction homomorphism. We record some useful facts:

\begin{proposition}\label{prop: monodromy}
Let $F$ be a hyperk\"ahler manifold of K3$^{[2]}$-type. Then:
\begin{enumerate}
    \item  $\Mon^2_{Hdg}(F)= W_{Exc}\rtimes \Mon_{Bir}^2$;
    \item $\Mov(F)$ is a fundamental domain for the action of $W_{Exc}$ on $\overline{\mathrm{Pos}(F)}$;
    \item The kernel of $\pi$ is a subgroup of $\Mon^2_{Aut}$, the set of monodromy operators induced by automorphisms of $F$.
\end{enumerate}
\end{proposition}
\begin{proof}
    The first statement is \cite[Theorem 6.18]{markman}. The second statement is \cite[Lemma 6.22]{markman}. For the third, \cite[Lemma 6.23]{markman} asserts that the kernel of $\pi$ is a subgroup of $\Mon_{Bir}^2.$ By \cite{Fujiki}, a birational map acting trivially on $\NS(F)$ is regular.
\end{proof}

\begin{remark}\label{remark:bir embeds}
    If $F$ is the Fano variety of lines on a cubic fourfold $X$ with $\Aut(X)=0$, then $\Bir(F) \to \Orth(\NS(F))$ is an embedding. Indeed, $\Aut(F,g)=\Aut(X)=0$, so no automorphisms of $F$ act trivially on $\NS(F)$ (see \cite[Proposition 4]{charles2012remark}). In particular, by Proposition~\ref{prop:Aut(X)=0}, this always applies for the cubic fourfolds we study. 
\end{remark}

\subsection{GM fourfolds and double EPW sextics}\label{subsec: GM4folds}
Later, we will exhibit birational models of the Fano variety $F$ of lines on a general cubic fourfold $X\in\calC_{12}$ as double Eisenbud-Popescu-Walter (EPW) sextics constructed from smooth Gushel--Mukai (GM) fourfolds. In this section we recall the relevant background and terminology.

Let $V_5$ be a complex vector space of dimension $5$.
\begin{definition}
    An (ordinary) GM fourfold $Z$ is a smooth transverse intersection of the form
\[
    Z \coloneqq \Gr(2,V_5) \cap \bP^8\cap Q\subset \bP(\bigwedge\nolimits^2V_5),
\]
    where $\Gr(2,V_5)$ is given the Pl\"ucker embedding, $\bP^8$ is a linear subspace, and $Q$ is a quadric.
\end{definition}
Such a fourfold is a Fano variety with Picard number $1$, index $2$ and degree $10$ \cite{MukaiFanoclassification}. In \cite{IM11}, the authors relate a GM fourfold $Z$ to an EPW sextic $W$ in the following way. Let $I\coloneqq |\calO_Z(2)|$ be the linear system of quadrics in $\bP^8$ containing $Z$, and let $\mathrm{Disc}(Z) \subset I$ be the irreducible component of the discriminant hypersurface that parametrizes singular quadrics that are not restrictions of the Pl\"ucker quadrics.

\begin{theorem}[{\cite[Proposition 2.4]{IM11}, \cite[Proposition 3.18]{DK1}}]
    Let $Z$ be an ordinary GM fourfold. The component $\mathrm{Disc}(Z) \subset I$ of the discriminant locus is an EPW sextic.
\end{theorem}

Taking an appropriate double cover of $W\coloneqq \mathrm{Disc}(Z)$, one obtains a double EPW sextic associated to $Z$ coinciding with the one constructed and studied by O'Grady in \cite{OG1, OG2, OG3}. We instead review the construction of the double EPW sextic dual to this double cover. Let $W^\vee\subset I^\vee$ be the hypersurface dual to $\mathrm{Disc}(Z)$; when $Z$ contains no decomposable vectors, $W^\vee$ is again an EPW sextic \cite[Corollary 3.6]{OG4} (see also \cite[Proposition B.3]{DK1}). We make use of the following geometric interpretation of $W^\vee$, explained in \cite{IM11} (see the construction of the map $\alpha$ preceding Proposition 4.9); see \cite{DK24} (specifically Sections 4 and 7.6) for a generalization.

\begin{lemma}\label{lemma:epw singular quadrics}
The EPW sextic $W^\vee$ satisfies the following properties:
\begin{enumerate}
    \item\label{item: point to quadric} Let $w\in W^\vee$ be a general point. Then $w$ corresponds to a unique singular quadric threefold $Q_w\subset\bP(\wedge^2V_4)\subset\bP^8$ for some $V_4\subset V_5$, and vice versa.
    \item\label{item: conic to quadric} Let $C\subset Z$ be a general conic not contained in any plane in $Z$. Then there is a unique point $w\in W^\vee$ such that $Q_w$ contains the plane spanned by $C$.
 \end{enumerate}
\end{lemma}
\begin{proof}
    Recall that there is a distinguished hyperplane $H_P\subset I$ spanned by the Pl\"ucker quadrics that contain $Z=\Gr(2,V_5)\cap \bP^8\cap Q\subset \bP(\wedge^2V_5).$
    A point $w\in W^\vee$ determines a distinct hyperplane $H\subset I$. The intersection $H\cap H_P\cong \bP^3=\bP(V_4)$ determines a $V_4\subset V_5$, after identifying the space of Pl\"ucker quadrics $H_P$ with $\bP(V_5)$. 
    Writing $H=\bP(V_4\oplus \langle\tilde{Q}\rangle)$ for some non-Pl\"ucker quadric $\tilde{Q}$, one finds that the base locus of $H$ when restricted to the ambient $\bP^8$ is the union of $Z$ and a quadric threefold $Q_w\subset \tilde{Q}\cap \bP^8$. 
    Since $H$ is tangent to $W\subset I$, $Q_w$ is singular. 
    Moreover, for general $w$, \cite[Proposition 4.7]{IM11} the hyperplane $V_4\subset V_5$ is unique, and $Q_w:=\tilde{Q} \cap\bP(\wedge^2V_4)\cap \bP^8\subset\bP(\wedge^2V_5)$, where $\tilde{Q}$ is a member of $H$.

    Conversely, suppose $Q'\subset\bP(\wedge^2V_4)$ is a singular quadric threefold. Then the complete linear system of quadrics containing $Q'$ determines a hyperplane $H\subset I$ which is tangent to $W$ since $Q'$ is singular. Hence we obtain a point $w\in W^\vee$, proving (\ref{item: point to quadric}).

    For (\ref{item: conic to quadric}), let $C\subset Z$ be a conic such that the plane $\langle C\rangle$ is not contained in $Z$. A general such conic $C$ is not a $\rho$-conic (see the discussion at the beginning of \cite[Section 4.4]{IM11}), so there is a unique $V_4\subset V_5$ such that $\bP(\wedge^2 V_4)$ contains $C$. Let $P_{V_4}$ denote the intersection of the Pl\"ucker quadric $G(2,V_4)$ with $\bP^8\subset \bP^9$, and let $Q_{V_4}\coloneqq Q\cap \bP(\wedge^2 V_4)\cap \bP^8$. 
    This gives a pencil $\langle P_{V_4}, Q_{V_4}\rangle$ of quadric threefolds in $\bP(\wedge^2 V_4)$, uniquely determined by $Z$ and $C$.
    Since $\langle C\rangle\not\subset Z$, there is a unique quadric $Q_{C}$ in the pencil containing this plane (for more details, see the discussion preceding Proposition 4.9 of \cite{IM11}). The quadric $Q_{C}$ is thus a singular quadric threefold, and by (\ref{item: point to quadric}) there exists a unique $w\in W^\vee$ such that $Q_{C}=Q_w.$
\end{proof}

\begin{remark}\label{rem:doubleEPWrulings}
In light of the above, a general point on the double EPW sextic $\widetilde{W}^\vee$ associated to $W^\vee$ can be regarded as a ruling on one of the quadric threefolds $Q_w$. The covering involution on $\widetilde{W}^\vee$ exchanges the rulings of each $Q_w$.    
\end{remark}

\section{Cubic fourfolds containing one cubic scroll}\label{sec:onescroll}

Let $X$ be a smooth cubic fourfold containing a cubic scroll $T$, and let $F$ be its Fano variety of lines. We assume $X$ is \emph{very general}, by which we mean $\rank(A(X))=2$, $\Aut(X)=0$, and $\mathrm{End_{Hdg}}(T(F))=\{\pm1\}$; the third assumption can be made in light of \cite[Section 14]{bayerfluckiger2024k3surfacesrealcomplex}. In \cite{flops}, the authors compute the ample and movable cone of $F$, and they exhibit a birational automorphism of $F$ of infinite order. They also describe hyperk\"ahler fourfolds birational to $F$ as follows. The scroll $T$ spans a hyperplane $H$, and by Proposition \ref{prop: P, S, Pv}, the Fano variety of lines on $Y\coloneqq X \cap H$ is a  subvariety of $F$ decomposing as \(F(Y)=P\cup S'\cup P^\vee\),
where $P$ and $P^\vee$ are Lagrangian planes. Let $F_1$ and $F_1^\vee$ be the Mukai flops of $F$ along $P$ and $P^\vee$, respectively.

We expand on their study using the more recent techniques outlined in \Cref{subsec: HK} for studying the birational geometry of hyperk\"ahler manifolds. More specifically, we prove the following: 

\begin{theorem}\label{mainthm:1scroll}
       Let $F$ be the Fano variety of lines of a very general member $X\in \calC_{12}$. Then $F$ has three isomorphism classes of birational hyperk\"ahler models, represented by $F$ itself and two non-isomorphic Mukai flops, both of which are isomorphic to double EPW sextics. Moreover, $\Bir(F)$ is generated by the covering involutions on these two double EPW sextics.
    \end{theorem}

In \Cref{subsec:NSFlattice}, we study isometries of $\NS(F)$, and in particular those induced by birational automorphisms of $F$. In \Cref{subsec:birmodels1scroll}, we provide a correction to Hassett and Tschinkel's count on the number of non-isomorphic birational hyperk\"ahler models, showing that $F$ has exactly three birational models up to isomorphism, represented by $F$, $F_1$, and $F_1^\vee$; we also complete their description of $\Bir(F)$. Finally, in \Cref{sec:epw},  we prove that $F_1$ (and similarly $F_1^\vee$) 
is in fact isomorphic to a double EPW sextic associated to the pair $(X,T^\vee)$ (similarly, the pair $(X, T)$).

\subsection{The Lattice $\NS(F)$}\label{subsec:NSFlattice}
 
Recall from Section~\ref{subsec: Fano prelims} that $\NS(F)$ is isomorphic to the lattice $J_{12}$. The discriminant group is $\bZ/6\times\bZ/4$, with factors generated by $[\frac g6]$ and $[\frac{\lambda}4]$. The BBF form represents $-10$ but not $-2$, so $F$ contains no prime exceptional divisors, and $\Mov(F)=\overline{\mathrm{Pos}(F)}$.

As noted in \cite[Section 7]{flops}, the isometry group of $\NS(F)$ is the product $\{\pm1\}\times\Gamma$, where $\Gamma$ is the infinite dihedral group
\[
\Gamma=\langle R_1,R_2\;|\; R_1^2=R_2^2=1\rangle;
\]
explicitly, the generators are
\[
R_1=\begin{pmatrix*}[r] 1 & 0 \\ 0 & -1 \end{pmatrix*}\;\; \mathrm{and} \;\; R_2=\begin{pmatrix*}[r] 5 & -4 \\ 6 & -5 \end{pmatrix*}.
\]
We determine which isometries are induced by birational automorphisms of $F$:

\begin{lemma}\label{lemma:actionondiscriminant}
    An isometry $\varphi\in \Orth(\NS(F))$ is induced by a birational automorphism of $F$ if and only if $\varphi$ preserves the positive cone and acts on the subgroup $H$ of $D_{\NS(F)}$ generated by $[\frac g3]$ and $[\frac\lambda4]$ by $\pm\id$.
\end{lemma}
\begin{proof}
Recall that $\Mov(F)=\overline{\mathrm{Pos}(F)}$, so if $\varphi \in \Orth(\NS(F))$ satisfies $\varphi=g^*$ for some $g \in \Bir(F)$, then \Cref{thm:torelli} implies that $\varphi$ preserves the positive cone. 
Moreover, $\varphi$  restricts to a Hodge isometry of $T(F)$, and since $F$ is very general, the only Hodge isometries of $T(F)$ are $\pm \id_{T(F)}$. Such an isometry necessarily acts on the discriminant group $H'\coloneqq D_{T(F)}$ by $\pm\id$.
The overlattice $H^2(F,\bZ)\supset \NS(F)\oplus T(F)$ corresponds to an index two subgroup $H\subset D_{\NS(F)},$ via Nikulin's theory of overlattices \cite[Propositions 1.4.1, 1.4.2]{nikulin}.
In particular, it follows that $(H, q_{\NS(F)}) \cong (H', -q_{T(F)}).$ Further, $D_{H^2(F,\bZ)}$ is generated by the class $[\frac{g}{2}]$, and $H\cong [\frac{g}{2}]^\perp\subset D_{\NS(F)}$. Thus $H$ is as claimed, and $\varphi$ acts on $H$ by $\pm \id$.

Conversely, if $\varphi$ preserves $\overline{\mathrm{Pos}(F)}=\Mov(F)$ and acts on $H$ by $\pm \id$, then by \Cref{thm:torelli} it suffices to show that $\varphi \in \im (\pi\colon \Mon^2_{Hdg}(F)\to \Orth(\NS(F)))$. The action of $\varphi$ on $H$ along with Nikulin's theory of overlattices implies that $\varphi$ extends to an isometry $\widetilde{\varphi}$ of $H^2(F,\bZ)$, acting as $\pm \id_{T(F)}$. Since by construction $\widetilde{\varphi}$ preserves $T(F)$, it preserves the Hodge structure, hence comes from $\Mon^2_{Hdg}(F)$. 
\end{proof}

\begin{lemma}\label{lemma:monodromygenerators}
    The subgroup of $\Orth(\NS(F))$ consisting of isometries induced by birational automorphisms of $F$ is generated by the reflections $R_2$ and $R_1R_2R_1$. Moreover, $\Bir(F)\cong\langle a,b\;|\; a^2=b^2=1\rangle$.
\end{lemma}
\begin{proof}
    The induced actions of $R_1$ and $R_2$ on the subgroup
    $
    \left\langle\left[\frac g3\right]\right\rangle \times\left\langle\left[\frac\lambda4\right]\right\rangle\subset D_{\NS(F)}
    $
    are given, respectively, by the matrices
    \[
    \begin{pmatrix*}[r] 1 & 0 \\ 0 & -1 \end{pmatrix*}
    \text{ and }
    \begin{pmatrix*}[r] -1 & 0 \\ 0 & -1\end{pmatrix*}.
    \]
    By Lemma~\ref{lemma:actionondiscriminant}, it is easy to see that an isometry of $\NS(F)$ is induced by a birational automorphism of $F$ if and only if, written as a word in $R_1$ and $R_2$, the generator $R_1$ appears an even number of times. Since the induced actions of $R_1$ and $R_2$ on $D_{\NS(F)}$ are involutions that commute, this proves the first claim.

    The second claim follows from the complete description of generators and relations on $\Gamma\subset\Orth(\NS(F))$ and the fact that $\Bir(F)$ embeds in $\Orth(\NS(F))$ by Remark~\ref{remark:bir embeds}.
\end{proof}

\begin{remark}\label{remark:iota definition}
    The isometries of $\NS(F)$ given by $R_1R_2R_1$ and $R_2$ are induced by birational involutions $\iota$ and $\iota^\vee$, becoming regular on $F_1$ and $F_1^\vee$, respectively (cf.~\cite[proof of Theorem~7.3]{flops} for computing $\iota^*$ and $(\iota^\vee)^*$). We review the geometric descriptions of these involutions, given in \cite[Theorem 6.2]{flops}. Recall that $T\subset X\subset \bP^5$ spans a hyperplane $H\subset\bP^5$, and $Y\coloneqq H\cap X$ has $F(Y)=P\cup S'\cup P^\vee$ (see Proposition \ref{prop: P, S, Pv}). 
    Given a line $[m]\in F\setminus F(Y)$, the point $p=H\cap m$ lies on unique lines $\ell$ and $\ell^\vee$ in the families $P$ and $P^\vee$ of $F(Y)$, respectively \cite[Corollary 4.2]{flops}. Let $\Pi_m=\Span\langle m,\ell\rangle$ and $\Pi_m^\vee=\Span\langle m,\ell^\vee\rangle$. Then the decompositions
    \[
    \Pi_m\cap X=m\cup\ell\cup\iota^\vee(m)
    \]
    and
    \[
    \Pi_m^\vee\cap X=m\cup\ell^\vee\cup\iota(m)
    \]
    define $\iota$ and $\iota^\vee$ away from $F(Y)$. In fact, $\iota$ extends over $P^\vee\setminus(P\cup S')$, and $\iota^\vee$ extends over $P\setminus(P^\vee\cup S')$.
    \end{remark}
    
\subsection{Census of birational models of $F$.}\label{subsec:birmodels1scroll}

The movable cone of $F$ contains infinitely many chambers, corresponding to the nef cones of birational models of $F$, as outlined in \cite[Theorem 7.4]{flops} and \Cref{subsec: HK}. To enumerate the chambers, we follow \cite[Section 7]{flops} by enumerating the wall divisors, i.e. $v\in\NS(F)$ with $v^2=-10$.  Let $\rho_1=g-2\lambda$, $\rho_2=3g-4\lambda$, $\rho_i^\vee=R_1(\rho_i)$, and for all integers $n$,
\[
\rho_{2n+i}=(R_1R_2)^n\rho_i
\]
where we interpret $\rho_{i}=\rho_{-i}^\vee$ for $i<0$. Using standard propagation techniques, one sees that all the classes $v$ with $v^2=-10$ and $q(v,g)\ge0$ are of the form $\rho_i$ or $\rho_i^\vee$ for some $i$. 

Let $\alpha_i$ be the class spanning $\rho_i^\perp$ and pairing positively with $g$; similarly, let $\alpha_i^\vee$ span $(\rho_i^\vee)^\perp$ and pair positively with $g$. Then the walls of $\Mov(F)$ are spanned by the $\alpha_i$ and $\alpha_i^\vee$.
In \cite[Proposition 7.2]{flops}, it is shown that $\Nef(F)=\Cone(\alpha_1,\alpha_1^\vee)$, and by \cite[Proposition 7.3]{flops}, we have $\Nef(F_1)=\Cone(\alpha_1,\alpha_2)$ and $\Nef(F_1^\vee)=\Cone(\alpha_1^\vee,\alpha_2^\vee)$. Moreover, for all $n$, $\Cone(\alpha_n,\alpha_{n+1})$ and $\Cone(\alpha_n^\vee,\alpha_{n+1}^\vee)$ are chambers of $\Mov(F)$ representing nef cones of birational models $F_n$ and $F_n^\vee$.

Essentially by construction,
\[
R_1R_2\cdot\Nef(F_n)=\Nef(F_{n+2}),
\]
so the authors of \cite{flops} conclude in Theorem 7.4 that $F_n\simeq F_{n+2}$ for all $n$ (again interpreting $F_n=F_{-n}^\vee$ when $n<0$). In particular, this would mean that $F$ has at most two birational models up to isomorphism, represented by $F$ and $F_1$. 
However, by Lemma~\ref{lemma:monodromygenerators}, we see that $R_1R_2$ is \textbf{not} induced by a birational automorphism of $F$ so by Theorem~\ref{thm:torelli} need not send the nef cone of one model to the nef cone of some isomorphic model. Instead, we provide the following correction to \cite[Theorem 7.4]{flops}:

\begin{proposition}\label{prop:model census}
    Up to isomorphism, $F$ admits exactly three birational hyperk\"ahler models, represented by $F$, $F_1$, and $F_1^\vee$.
\end{proposition}
\begin{proof}
    First, note $(R_1R_2)^2=\iota^*\circ(\iota^\vee)^*$ by Remark~\ref{remark:iota definition}, so from
    \[
    (R_1R_2)^2\cdot\Nef(F_n))=\Nef(F_{n+4}),
    \]
    we deduce $F_n\simeq F_{n+4}$ for all $n$. Moreover, it is straightforward to calculate
    \[
    \iota^*\Nef(F)=\Nef(F_2)\quad\quad \text{and} \quad\quad     (\iota^\vee)^*\Nef(F)=\Nef(F_2^\vee),
    \]
   
    so $F_2\simeq F\simeq F_2^\vee$, by Remark~\ref{remark: when bir is reg}. It follows that $F$ has at most three birational hyperk\"ahler models up to isomorphism, represented by $F$, $F_1$, and $F_1^\vee$. We now show that these three are non-isomorphic. 

    To distinguish $F$ from $F_1$ and $F_1^\vee$, note that $F_1$ and $F_1^\vee$ both admit nontrivial involutions whereas $F$ does not: indeed, the only nontrivial automorphism of $\Nef(F)$ preserving the positive cone is $R_1$ which, by Lemma~\ref{lemma:monodromygenerators}, is not induced by a birational automorphism of $F$. To distinguish $F_1$ from $F_1^\vee$, note that the only two isometries of $\NS(F)$ sending $\Nef(F_1)$ to $\Nef(F_1^\vee)$ are $R_1$ and $R_2R_1$. Again, by Lemma~\ref{lemma:monodromygenerators}, neither of these lattice automorphisms are induced by a birational automorphism of $F$.
\end{proof}

\subsection{Connection with Gushel--Mukai fourfolds}\label{sec:epw}

Here, we show that $F_1$ and $F_1^\vee$ are isomorphic to  double EPW sextics. This relates each birational model of $F$ to a familiar family of hyperk\"ahler fourfolds and explains the birational involutions $\iota$ and $\iota^\vee$ on $F$---they are induced by the covering involutions associated to the double EPW sextics.

Since $X$ is a smooth cubic fourfold containing a smooth cubic scroll $T$, we can apply the following construction, due to \cite[Section 7.2]{DIM} and \cite[Proposition 5.6]{KuzPerry}. The complete linear system of quadrics containing the scroll $T$ induces a rational map $q\colon X\dashrightarrow\bP^8$ which is birational onto its image, a GM fourfold $Z_T\subset \Gr(2,V_5)$ containing a plane $\Pi$.  Projection from $\Pi$ induces a rational inverse $f\colon {Z_T}\dashrightarrow X$. We obtain the following diagram, where $Y^+\to Y$ is the small resolution from~\eqref{diag:smallresolution}, and $E$ is the exceptional divisor.

\begin{center}
    \begin{tikzcd}
        Y^+ \arrow[d] \arrow[rrr,"\sim"] \arrow[dr,hookrightarrow]  & & & E \arrow[dl,hookrightarrow] \arrow[d]\\
         Y  \arrow[d,hookrightarrow]  & \Bl_TX \arrow[dl] \arrow[r,"\sim"] & \Bl_\Pi Z_T \arrow[dr] & \Pi \arrow[d,hookrightarrow] \\ 
         X \arrow[rrr,dashed,"q"] & & & Z_T 
    \end{tikzcd}
\end{center}

Note that the isomorphism $\Bl_TX\simeq \Bl_\Pi Z_T$ identifies $Y^+$ with $E$, so the image of $Y^+$ in $Z_T$ is $\Pi$.

Let $W$ be the EPW sextic associated to $Z_T$ and $W^\vee$ its dual, as defined in \Cref{subsec: GM4folds}. While $Z_T$ depends on the choice of scroll $T$ in its homology class, $W$ does not \cite[Proposition 7.2]{DIM}. Let $\widetilde{W}^\vee$ be the double cover of $W^\vee$, also introduced in \Cref{subsec: GM4folds}, equipped with its covering involution $\tau$.

\begin{lemma}\label{lemm:QcapPi}
    Let $Q_w$ be a singular quadric threefold associated via~\Cref{lemma:epw singular quadrics} to a point $w\in W^\vee$. Then either $\Pi\subset Q_w$ or $\Pi$ and $Q_w$ meet in a point. For general $w$, the quadric $Q_w$ is unique, and $\Pi$ and $Q_w$ meet in one point which is not the cone point of $Q_w$.
\end{lemma}
\begin{proof}
    By \cite[Section 7.2]{DIM}, $\Pi=\bP(\wedge^2V_3)\subset Z_T$ for some three-dimensional space $V_3\subset V_5$. For general $Q_w$, there is a unique hyperplane $V_4\subset V_5$ such that $Q_w= \bP(\wedge^2V_4)\cap \tilde{Q}$ for a (non Pl\"ucker) quadric $\tilde{Q}\subset \bP^8$ containing $Z_T$ by Lemma~\ref{lemma:epw singular quadrics}(\ref{item: point to quadric}). 
    We have $\Pi\cap Q_w=\Pi\cap\bP(\wedge^2V_4)\cap \tilde{Q}$ and note that $\tilde{Q}$ contains $Z_T$ and thus $\Pi.$ Hence $\Pi\cap Q_w=\bP(\wedge^2(V_3\cap V_4))$ which is either a point or all of $\Pi$ depending on whether $V_3\cap V_4$ is dimension two or three. Only a two-dimensional family of hyperplanes in $V_4$ contain $V_3$. Moreover, by the proof of \cite[Lemma 2.3]{IM11}\footnote{In Section 2 of \cite{IM11}, the authors include a generality assumption on the quadric $Q$ used to define the GM fourfold $Z_T$ and the double EPW sextic $\widetilde{W}^\vee$ to ensure that $Z_T$ is smooth and of the correct dimension; here, that generality assumption is satisfied.}, each point in $\bP^8$ is the cone point of $Q_w$ for at most one $w\in{W}^\vee$; in particular, the cone point of $Q_w$ belongs to $\Pi$ for at most a two-dimensional locus in ${W}^\vee$.
\end{proof}

\begin{proposition}
    $F$ and $\widetilde{W}^\vee$ are birational.
\end{proposition}
\begin{proof}
    We define a rational map $\beta\colon F\dashrightarrow\widetilde{W}^\vee$ as follows. 
    Recall that $H\subset \bP^5$ is the hyperplane spanned by $T$.
    A general line $m\subset X\subset \bP^5$ meets $H$ in a point $p$, and we claim that $C_m=q(m)$ is a conic in $Z_T$ meeting $\Pi$ in a point.
    Indeed, $m$ does not pass through a singular point of $Y$ (or else we would have $m\subset Y)$, so the strict transform of $m$ in $\Bl_TX$ meets $Y^+$ in a point. 
    Under the identification $Y^+\cong E$, we see that the strict transform of $m$ meets the exceptional divisor of $\Bl_\Pi Z_T\to Z_T$ in a point; hence $C_m$ meets $\Pi$ in a point. 
    By \cite[Proposition 3.7]{DK24} $Z_T$ contains at most finitely many planes other than $\Pi$, each of which pulls back to a surface in $X$, the image of a general line $m\subset X$ is not contained in any plane of $Z_T$.
     By \Cref{lemma:epw singular quadrics}(\ref{item: conic to quadric}) there exists a unique singular quadric $Q_m\subset \bP(\wedge^2V_4)$ for some $V_4\subset \bP^5$ containing the plane spanned by $C_m$. Thus $C_m$ determines a ruling of $Q_m$ and hence a point $\beta(m)\in\widetilde{W}^\vee$ by Remark~\ref{rem:doubleEPWrulings}.

    To show $\beta$ is a birational equivalence, we describe its inverse. A general point $w\in \widetilde{W}^\vee$ specifies a ruling on a singular quadric threefold $Q_w$. Lemma~\ref{lemm:QcapPi} implies there is a unique plane $P_w$ in that ruling of $Q_w$ containing the point $Q_w\cap \Pi$. The plane $P_w$ meets $Z_T$ in a conic $C_w$ intersecting $\Pi$, so projection from $\Pi$ yields a line in $X$, which we take as $\beta^{-1}(w)$. One can check that this construction is inverse to the construction of $\beta$.
\end{proof}

Let us briefly mention how the rational map $\beta$ we constructed relates to recent work in \cite{DK24}, generalizing a construction from \cite{IM11}. In \cite[Theorem 7.12]{DK24}, the authors prove that the Hilbert scheme of conics on $Z_T$ has multiple irreducible components, one for each plane in $Z_T$ and another denoted $\overline{G_1^0(Z_T)}$. They moreover show that the component $\overline{G_1^0(Z_T)}$ is birational to a fivefold $G_1^+(Z_T)$ admitting a morphism $f^+\colon G_1^+(Z_T)\to \widetilde{W}^\vee$ whose general fiber is $\bP^1$. Since $q$ sends a general line on $X$ to a conic on $Z_T$ not lying on any plane of $Z_T$, it induces a rational map $F\dashrightarrow G_1^+(Z_T)$. The arguments above imply that the image of this rational map is a rational section of $f^+$.

\begin{proposition}\label{prop:flop is epw}
    The birational map $\beta$ induces an isomorphism $F_1^\vee\simeq\widetilde{W}^\vee$.
\end{proposition}

\begin{proof}
    It suffices to show that $\iota^\vee$ becomes regular on $\widetilde{W}^\vee$, and in particular, we show $\beta\circ\iota^\vee\circ\beta^{-1}$ and the covering involution $\tau$ agree on an open set. 
    
    Let $w\in W^\vee$ be a point defining a singular quadric threefold $Q_w$. By Lemma~\ref{lemm:QcapPi}, for general $Q_w$ we have $Q_w\cap\Pi=\{x\}$ where $x$ is not the cone point of $Q_w$. Note also that $x$ corresponds to the 2-dimensional subspace $U_2:=V_3\cap V_4$; indeed $x=\bP(\wedge^2U_2)\subset\bP(\wedge^2V_4)\cap \bP^8$. Hence there is a unique plane in each ruling of $Q_w$ containing $x$; these planes intersect $Z$ in conics $C$ and $\tau(C)$. The conics $C$ and $\tau(C)$ intersect in two points, counting multiplicity---once at $x$ and again along the line joining $x$ to the cone point of $Q_w$.

    First, consider the case where $C\cap\tau(C)$ consists of distinct points $x$ and $y$. Note that $y\not\in\Pi$. Projecting from $\Pi$, we obtain lines $m=q^{-1}(C)$ and $m'=q^{-1}(\tau(C))$ intersecting at $q^{-1}(y)\not\in H$. Moreover, by Remark~\ref{remark:lines in T} the fiber of $Y^+\to\Pi$ over $x$ is a line whose image $\ell$ in $Y$ is contained in a cubic scroll $T'$ homologous to $T^\vee$, and $\ell^2=0$ on $T'$. Since any two of $m$, $m'$, and $\ell$ intersect, but not all at the same point, the three lines are coplanar.

    Alternatively, $C$ and $\tau(C)$ are tangent at $x$. In that case, the total transform of $C$ and $\tau(C)$ under the projection $\Bl_\Pi Z_T\to X$ again consists of a  triple of lines $m=q^{-1}(C)$, $m'=q^{-1}(\tau(C))$, and a line $\ell\subset Y$ whose image under $Y^+\to\Pi$ is $x$. The common intersection point of these three lines is the point on $\ell$ corresponding to the normal direction to $\Pi$ at $x$ given by the shared tangent directions of $C$ and $\tau(C)$ at $x$. As before, $\ell$ is contained in a cubic scroll $T'$ homologous to $T^\vee$, and $\ell^2=0$ on $T'$. Moreover, the three lines $m$, $m'$, and $\ell$ are coplanar, since $x$ is an Eckardt point of the cubic surface obtained by intersecting $X$ with the $\bP^3$ obtained by projecting $\langle Q_w\rangle$ from $x$.
    
    In either case, $[\ell]\in P$ by  \cite[Proposition 4.7]{flops}, and $m$, $m'$, and $\ell$ are coplanar. Therefore, recalling the definition of $\iota^\vee$ from Remark~\ref{remark:iota definition}, we have $\iota^\vee(m)=m'$. In other words, $\iota^\vee(\beta^{-1}([C]))=\beta^{-1}(\tau([C]))$, proving the claim.
\end{proof}

Starting instead with $\Bl_{T^\vee}X$, one obtains an isomorphism between $F_1$ and a double EPW sextic whose covering involution induces $\iota$. By Lemma~\ref{lemma:monodromygenerators} and Remark~\ref{remark:iota definition}, the covering involutions on the two double EPW sextics generate $\Bir(F)$. This observation, together with Propositions~\ref{prop:model census} and Proposition~\ref{prop:flop is epw}, completes the proof of Theorem~\ref{mainthm:1scroll}.

\begin{remark}
    It would be interesting to know whether $Z_T$ and $Z_{T^\vee}$ are dual GM fourfolds, or equivalently if the double EPW sextics associated as above to $(X,T)$ and $(X,T^\vee)$ are dual (see \cite[Definition~3.26]{DK1}). By \cite[Theorem~1.6]{KuzPerryCatCones}, dual GM fourfolds are Fourier--Mukai partners. Kuznetsov and Perry in \cite[Theorem~5.8]{KuzPerry} show directly that $X$ and $Z_T$ have derived equivalent Kuznetsov components, which implies the same is true of $Z_T$ and $Z_{T^\vee}$.  
\end{remark}

\section{A syzygetic pair of cubic scrolls}\label{sec:syzygetic}

Let $X\subset \bP^5$ be a smooth cubic fourfold containing a syzygetic pair of cubic scrolls $T_1$ and $T_2$. We also assume $X$ is \emph{very general}, meaning $\rank(A(X))=3$, $\Aut(X)=0$, and $\mathrm{End_{Hdg}}(T(F))=\{\pm1\}$. In this section, we completely describe the birational geometry of the Fano variety $F$ of lines on $X$. Specifically, we prove the following:

\begin{theorem}\label{thm:syzmain}
    Let $F$ be the Fano variety of lines on a very general cubic fourfold $X$ containing a syzygetic pair of cubic scrolls $T_1, T_2$. Then $F$ has five isomorphism classes of birational hyperk\"ahler models, represented by $F$ itself and four non-isomorphic Mukai flops of $F$. Moreover, each of the four Mukai flops of $F$ can be realized as a double EPW sextic.
\end{theorem}

Furthermore, we determine the birational automorphism group of $F$:

\begin{theorem}\label{theo:birFsyz}
    Let $F$ be as above. Then 
    \[
    \Bir(F)\cong\langle a,b,c,d\;|\; a^2=b^2=c^2=d^2=1\rangle.
    \]
    Moreover, the four generators can be identified with the covering involutions on the double EPW sextics obtained as Mukai flops of $F$.
\end{theorem}

The outline is as follows: in Section~\ref{subsec:syzygetic planes}, we study the arrangement of Lagrangian planes in $F$ and show that any two planes in $F$ parametrizing lines contained in the hyperplanes spanned by 
$T_1$ and $T_2$ must intersect. 
We then identify which isometries of $\NS(F)$ are induced by birational automorphisms of $F$ in Section~\ref{subsec:syzygetic lattice}, including identifying how the involutions from Remark~\ref{remark:iota definition} act. After enumerating the (infinitely many) walls of the movable cone of $F$ in Section~\ref{subsec:syzygetic walls}, we enumerate the birational models and identify the birational automorphism group in Section~\ref{subsec:syzygetic models}. 

\subsection{Lagrangian planes in $F$}\label{subsec:syzygetic planes}

Let $H_i$ be the hyperplane spanned by $T_i$, and $Y_i=X\cap H_i$. Recall from Section~\ref{subsection:threefolds} that 
\[
F(Y_i)=P_i\cup S_i'\cup P_i^\vee
\]
where $P_i$ and $P_i^\vee$ are Lagrangian planes. The intersection $F(Y_1)\cap F(Y_2)$ parametrizes lines on the cubic surface $\Sigma=X\cap H_1\cap H_2$. For a general cubic fourfold $X$ satisfying the hypotheses above, $\Sigma$ is smooth by Lemma~\ref{lem:Sigmasmooth}. The following lemma implies the intersection $F(Y_1)\cap F(Y_2)$ is transverse in general.

\begin{lemma}\label{lemma:transverse}
    Let $X$ be a smooth cubic fourfold containing a line $L$, let $F$ be the Fano variety of lines on $X$, and let $H_1$ and $H_2$ be distinct hyperplanes containing $L$. Suppose further the threefolds $Y_i=X\cap H_i$ and the surface $\Sigma=Y_1\cap Y_2$ are all smooth along $L$. Then $F(Y_1)$ and $F(Y_2)$ in $F$ intersect transversely at $[L]$.
\end{lemma}
\begin{proof}
    Since the intersection $\Sigma=Y_1\cap Y_2$ is transverse, and since $\Sigma$, $Y_i$, and $X$ are smooth along $L$, there is a short exact sequence of normal bundles 
    \[
    0\to N_{L/\Sigma} \to N_{L/Y_1}\oplus N_{L/Y_2} \to N_{L/X}\to0.
    \]
    Moreover, $N_{L/\Sigma}\simeq\mathcal{O}_L(-1)$ since $\Sigma$ is a cubic surface, so the long exact sequence in cohomology yields an isomorphism
    \[
    H^0(L,N_{L/Y_1})\oplus H^0(L,N_{L/Y_2}) \cong H^0(L,N_{L/X}).
    \]
    Identifying the cohomology groups above with tangent spaces, we find that 
    \[
    T_{[L]}F(Y_1)\oplus T_{[L]}F(Y_2)\cong T_{[L]}F(X),
    \]
    as needed.
\end{proof}

We also compare intersections on $X$ to intersections on $\Sigma$:

\begin{lemma}\label{lemma:gammaintersection}
    Let $X$ be as above, and let $y_i\in Y_i$ be a node. Further, let $A_i\cup A_i^\vee$ be the surface swept out by lines on $X$ through $y_i$, with $[A_i]=[T_i]$ and $[A_i^\vee]=[T_i^\vee]$ as in Remark~\ref{remark:fanodescription}, and let $\gamma_i=A_i\cap \Sigma$ and $\gamma_i^\vee=A_i^\vee\cap \Sigma$. Then $\gamma_i$ and $\gamma_i^\vee$ are twisted cubic curves, and $[T_1]\cdot[T_2]=[\gamma_1]\cdot[\gamma_2]$ where the first intersection pairing happens on $X$ and the second on $\Sigma$.
\end{lemma}
\begin{proof}
Since by Lemma~\ref{lem:Sigmasmooth} $\Sigma$ is smooth, none of the nodes of $Y_1$ lie on $Y_2$. Hence the linear section $\gamma_1=A_1\cap H_2$ does not pass through the cone point of $A_1$, so it is a smooth twisted cubic curve. The same is true of $\gamma_2$, $\gamma_1^\vee$, and $\gamma_2^\vee$. 

For the intersection numbers, we use the following cartesian square
    \[\xymatrix{\Sigma \ar[r]^{j_1} \ar[d]^{j_2} & Y_1 \ar[d]^{i_1} \\ Y_2 \ar[r]^{i_2} & X }\]
    and note that $[\gamma_i] = j_i^*[T_i]$. Using the base change formula, we have
    \[[\gamma_1]\cdot [\gamma_2] = j_1^*[T_1]\cdot j_2^*[T_2]= j_{2*}j_1^*[T_1]\cdot [T_2]= i_2^*i_{1*}[T_1]\cdot[T_2]= i_{1*}[T_1]\cdot i_{2*}[T_2]= [T_1]\cdot[T_2],\]
    as claimed.
\end{proof}

Recall from Remark~\ref{remark:fanodescription} that a line $\ell\subset Y_i$ is bisecant to $A_i$ or $A_i^\vee$ if and only if $[\ell]\in P_i$ or $P_i^\vee$, respectively. Conversely, $[\ell]\in S_i'$ if and only if $\ell$ meets both $A_i$ and $A_i^\vee$. So, one can tell which components of $F(Y_1)$ and $F(Y_2)$ a line $\ell\subset\Sigma$ lies in by looking at the incidence relations between $\ell$ and the curves $\gamma_i$ and $\gamma_i^\vee$. 

\begin{proposition}\label{proposition:syzygeticplanes}
Let $X$ be a cubic fourfold as above with $[T_1]\cdot[T_2]=3$. Then
    \begin{itemize}
        \item $\deg(P_1\cap P_2)=\deg(P_1\cap P_2^\vee)=\deg(P_1^\vee\cap P_2)=\deg(P_1^\vee\cap P_2^\vee)=1$,
        \item $\deg(P_1\cap S_2')=\deg(P_1^\vee\cap S_2')=\deg(S_1'\cap P_2)=\deg(S_1'\cap P_2^\vee)=4$,
        \item $\deg(S_1'\cap S_2')=7$,
    \end{itemize}
    and these intersections are all transverse when $\Sigma=X\cap H_1\cap H_2$ is smooth.
\end{proposition}

\begin{proof}
    It suffices to calculate the intersection degrees for any cubic fourfold satisfying the hypotheses, so by Lemma~\ref{lem:Sigmasmooth} we may assume that $\Sigma$ is smooth. In particular, $\Sigma$ being smooth requires that $H_1$ contains none of the nodes of $Y_2$, and $H_2$ contains none of the nodes of $Y_1$. Then, as mentioned in Lemma~\ref{lemma:transverse}, $F(Y_1)$ and $F(Y_2)$ intersect transversely. Hence we need only count points in the intersections set-theoretically.
    
    As in Lemma~\ref{lemma:gammaintersection}, let  $\gamma_i=A_i\cap \Sigma$ and $\gamma_i^\vee=A_i^\vee\cap \Sigma$. 
    By Remark~\ref{remark:fanodescription}, $\ell$ is bisecant to $\gamma_i$ or $\gamma_i^\vee$ if and only if $[\ell]\in P_i$ or $P_i^\vee$, respectively; $\ell$ meets both $\gamma_i$ and $\gamma_i^\vee$ if and only if $[\ell]\in S_i'$. Since $A_1\cup A_1^\vee=X\cap Q$ for some quadric $Q$, we see that 
    \[
    \gamma_1\cup\gamma_1^\vee=(A_1\cap H_1)\cup (A_1^\vee\cap H_1)\sim\Sigma\cap Q\sim-2K_\Sigma.
    \]
    As noted in \cite{dolgnodal}, whose notation for the lines on a cubic surface we adopt, the linear system $|\gamma_1|$ defines a morphism $f\colon \Sigma\to\bP^2$ blowing down a sixer $E_1,\dots,E_6$. Writing $E_0=f^*\mathcal{O}(1)$, the calculation above yields
    \[
    \gamma_1\sim E_0\;\;\text{and}\;\;\gamma_1^\vee\sim 5E_0-2\sum_{i=1}^6E_i.
    \]
 Now, write $\gamma_2\sim \sum_{i=0}^6a_iE_i.$ By Lemma~\ref{lemma:gammaintersection}, $a_0=[\gamma_1]\cdot [\gamma_2]=3$. Similarly,
    \[
    3=[\gamma_1^\vee]\cdot[\gamma_2]=15+2\sum_{i=1}^6a_i.
    \]
Since $\gamma_2$ is a twisted cubic, we also have $1=\gamma_2^2=9-a_1^2-\dots-a_6^2,$ so, possibly after relabeling $E_1,\dots,E_6$, we obtain
    \[
    \gamma_2\sim3E_0-2E_1-E_2-E_3-E_4-E_5.
    \]
    From here, the claim follows readily using the incidence relations for lines on $\Sigma$.
\end{proof}

In particular, since $P_1\cap P_1^\vee$ and $P_2\cap P_2^\vee$ are nonempty (see for example \cite[Proposition 4.6]{flops}), any two of the four Lagrangian planes in $F$ studied here intersect. This clarifies what we prove later: after flopping any one of the planes in $F$, one can no longer flop any of the other three.

\subsection{The lattice $\NS(F)$}\label{subsec:syzygetic lattice}

Letting $\lambda_i=\alpha(T_i-\eta_X)$, recall from Section~\ref{subsec: Fano prelims} that $\NS(F)$ with the BBF form is isomorphic to the lattice $J_{syz}$ and has discriminant group $D_{\NS(F)}\cong\bZ/6\times(\bZ/4)^2$ with factors generated by $[\frac g6]$, $[\frac{\lambda_1}4]$, and $[\frac{\lambda_2}4]$. The BBF form represents both $-2$ and $-10$.

The isometry group of $\NS(F)$ with basis $\{g,\lambda_1,\lambda_2\}$ is generated by $\pm1$ and the four reflections below. We calculated the isometry group $\Orth(\NS(F))$ via the algorithm outlined in \cite{Mertens} and implemented using the Magma package \texttt{AutHyp.m}.
\[
R_1=\begin{pmatrix*}[r] 5 & 4 & 0 \\ -6 & -5 & 0 \\ 0 & 0 & -1 \end{pmatrix*},\hspace{0.2cm} R_2=\begin{pmatrix*}[r] 1 & 0 & 0 \\ 0 & -1 & 0 \\ 0 & 0 & 1 \end{pmatrix*},\hspace{0.2cm} \]

\[R_3=\begin{pmatrix*}[r] 1 & 0 & 0 \\ 0 & 0 & 1 \\ 0 & 1 & 0 \end{pmatrix*},\hspace{0.2cm} R_4=\begin{pmatrix*}[r] 7 & -4 & 4 \\ 6 & -3 & 4 \\ -6 & 4 & -3 \end{pmatrix*}.
\]

The following lemma helps distinguish which isometries of the lattice $\NS(F)$ are induced by birational automorphisms of $F$.

\begin{lemma}\label{lemma:actionondisc}
    Any isometry $\varphi\in \Orth(\NS(F))$ is induced by a birational automorphism of $F$ if and only if $\varphi$ preserves $\Mov(F)$ and acts on the index two subgroup
    \[
    H\coloneqq \left\langle\left[\frac{g}3\right]\right\rangle\times\left\langle\left[\frac{\lambda_1}4\right]\right\rangle\times\left\langle\left[\frac{\lambda_2}4\right]\right\rangle
    \]
    of $D_{\NS(F)}$ by $\pm\id$.
\end{lemma}
\begin{proof}
    This is similar to \Cref{lemma:actionondiscriminant}, the only difference being that preserving $\overline{\mathrm{Pos}(F)}$ is no longer equivalent to preserving $\Mov(F)$.
\end{proof}

Let $F_i$ and $F_i^\vee$ denote the flops of $F$ along $P_i$ and $P_i^\vee$, respectively. As mentioned in Remark~\ref{remark:iota definition}, each of these models admits a regular involution: $\iota_i$ and $\iota_i^\vee$, respectively. These involutions can also be regarded as birational involutions on $F$.

\begin{lemma}\label{lemma:syzygeticiotas}
    The birational involutions $\iota_i$ and $\iota_i^\vee$ for $i=1,2$ act on $\NS(F)$ by $\iota_1^*=R_1$, $(\iota_1^\vee)^*=R_2R_1R_2$, $\iota_2^*=R_3R_1R_3$, and $(\iota_2^\vee)^*=R_3R_2R_1R_2R_3$.
\end{lemma}
\begin{proof}
    We give the proof for $\iota_1^*$. By \cite[Proposition~6.5]{flops}, $\iota_1^*$ fixes the class $g-\lambda_1$. Using this fact, along with the properties that $\iota_1^*$ preserves the BBF form and $(\iota_1^*)^2=\id$, a direct computation verifies that the only involutions of $\NS(F)$ fixing $g-\lambda_1$ are given by the matrices
    \[
    \begin{pmatrix*}[r] 1 & 0 & 0\\ 0 & 1 & 0 \\ 0 & 0 & \pm1 \end{pmatrix*}\;\text{and}\;\begin{pmatrix*}[r] 5 & 4 & 0\\ -6 & -5 & 0 \\ 0 & 0 & \pm1 \end{pmatrix*}.
    \]
    The first two candidates are impossible: they preserve $g$ and therefore preserve $\Nef(F)$ whereas $\iota_1$ is not regular on $F$. To distinguish between the two remaining candidates, we inspect actions on the discriminant group of $\NS(F)$. The matrix
    \[
    \begin{pmatrix*}[r]  5 & 4 & 0\\ -6 & -5 & 0 \\ 0 & 0 & \pm1  \end{pmatrix*}
    \]
    acts on $\langle[\frac g3]\rangle\times\langle[\frac{\lambda_1}4]\rangle\times\langle[\frac{\lambda_2}4]\rangle$ by
    \[
    \begin{pmatrix*}[r] -1 & 0 & 0 \\ 0 & -1 & 0 \\ 0 & 0 & \pm1 \end{pmatrix*}.
    \]
    By Lemma~\ref{lemma:actionondisc}, the sign in the last entry of this matrix must be negative, yielding the desired result. The other three actions are calculated similarly, using that $\iota_2^*$ fixes $g-\lambda_2$ and that $(\iota_i^\vee)^*$ fixes $g+\lambda_i$.
\end{proof}

\subsection{Wall and chamber structure of $\Mov(F)$}\label{subsec:syzygetic walls}

In order to study the birational geometry of $F$ and, in particular, to enumerate the birational hyperk\"ahler models of $F$ up to isomorphism, we give a description of the wall and chamber decomposition of the movable cone of $F$. 

As before, we fix the orthogonal basis $\{g,\lambda_1,\lambda_2\}$ on $\NS(F)$, and we denote a class 
\[
v=ag+b\lambda_1+c\lambda_2\in \NS(F)\otimes \bR
\]
by the vector $v=(a,b,c)$.

\begin{lemma}\label{lemma:nef}
    The nef cone of $F$ is bounded by the four walls $(1,\pm2,0)^\perp$ and $(1,0,\pm2)^\perp$. The only other walls in $\Mov(F)$ intersecting $\Nef(F)$ are the walls $(1,1,1)^\perp$, $(1,1,-1)^\perp$, $(1,-1,1)^\perp$, and $(1,-1,-1)^\perp$, each of which intersects $\Nef(F)$ in codimension two.
\end{lemma}
\begin{proof}
By \cite[Proposition~7.2]{flops}, we know that each of the walls $(1,\pm2,0)^\perp$ and $(1,0,\pm2)^\perp$ induces a small contraction of $F$, hence lies on the boundary of the nef cone. We prove that the only walls intersecting the chamber enclosed by the above walls are $(1,1,1)^\perp$, $(1,1,-1)^\perp$, $(1,-1,1)^\perp$, and $(1,-1,-1)^\perp$.
    
    We embed $\NS(F)\otimes\bR$ into $\bR^3$ using our basis and call the coordinate functions $x$, $y$, and $z$. Consider the cross-section of $\Mov(F)$ by the plane $x=4$: the cross-section of $\Cone((1,\pm2,0)^\perp,(1,0,\pm2)^\perp)$ is the square bounded by the four lines $y=\pm3$ and $z=\pm3$. In this slice, any wall of $\Nef(F)$ is a line cutting through this square so a point on the line has Euclidean distance at most $\sqrt{18}$ from the origin in this $yz$-plane.

    First, suppose $v=(a,b,c)$ satisfies $v^2=-2$. Any point on the line 
    \[
    (v^\perp)\cap\{x=4\}=\{6a-by-cz=0\}
    \]
    has Euclidean distance 
    \[
    \frac{6|a|}{\sqrt{b^2+c^2}}=\frac{6\sqrt2|a|}{\sqrt{3a^2+1}}
    \]
    from the origin in this $yz$-plane. This distance is no more than $\sqrt{18}$ if and only if $|a|\le1$. The only four such classes with $|a|\le1$ are $(1,1,1)^\perp$, $(1,1,-1)^\perp$, $(1,-1,1)^\perp$, and $(1,-1,-1)^\perp$, which intersect the boundary of $\Cone((1,\pm2,0)^\perp,(1,0,\pm2)^\perp)$.

    A similar calculation shows that no other classes with $v^2=-10$ intersect $\Cone((1,\pm2,0)^\perp,(1,0,\pm2)^\perp)$.
\end{proof}

We observe a nice geometric consequence:
\begin{corollary}
    The four planes in $F$ coming from components of $F(Y_i)$ where $Y_i=X\cap H_i$ for $i=1,2$ are the only Lagrangian planes in $F$.
\end{corollary}

Next, we describe the chambers of $\Mov(F)$ neighboring $\Nef(F)$.

\begin{lemma}\label{lemma:nefFisyzygetic}
The nef cones of the four flops of $F$ are 
\begin{align*}
    \Nef(F_1)&=\Cone((1,-2,0)^\perp,(1,-1,1)^\perp,(1,-1,-1)^\perp,(3,-4,0)^\perp), \\
    \Nef(F_1^\vee)&=\Cone((1,2,0)^\perp,(1,1,1)^\perp,(1,1,-1)^\perp,(3,4,0)^\perp), \\
    \Nef(F_2)&=\Cone((1,0,-2)^\perp,(1,1,-1)^\perp,(1,-1,-1)^\perp,(3,0,-4)^\perp), \\
    \Nef(F_2^\vee)&=\Cone((1,0,2)^\perp,(1,1,1)^\perp,(1,-1,1)^\perp,(3,0,4)^\perp).
\end{align*}
In particular, the chamber $\Nef(F_i)$ shares a face with two other chambers of $\Mov(F)$: $\Nef(F)$ and $\iota_i^*\Nef(F)$. Similarly, $\Nef(F_i^\vee)$ shares a face with the two chambers $\Nef(F)$ and $(\iota_i^\vee)^*\Nef(F)$.
\end{lemma}
\begin{proof}
    First, we give the proof for $F_1$. By \cite[Proposition 7.2]{flops}, the wall  $(1,-2,0)^\perp$ of $\Nef(F)$ corresponds to the small contraction of $F$ along $P_1$, and $F_1$ is the flop of $F$ along $P_1$, so $(1,-2,0)^\perp$ lies on the boundary of $\Nef(F_1)$. Moreover, $\iota_1$ is regular on $F_1$ so acts on $\Nef(F_1)$ by an involution; using Lemma~\ref{lemma:syzygeticiotas}, we calculate $\iota_1^*(1,-2,0)=(-3,4,0)$, so $(3,-4,0)^\perp$ must also lie on the boundary of $\Nef(F_1)$. Since the prime exceptional classes cut the movable cone out of the positive cone, we have
    \[
    \Nef(F_1)\subset\Cone((1,-2,0)^\perp,(1,-1,1)^\perp,(1,-1,-1)^\perp,(3,-4,0)^\perp).
    \]
    Using Lemma~\ref{lemma:nef}, we see that any other wall of $\Nef(F_1)$ intersecting the interior of the cone above would also cut into $\Nef(F)$ or $\iota_1^*\Nef(F)$, a contradiction. Hence 
     \[
    \Nef(F_1)=\Cone((1,-2,0)^\perp,(1,-1,1)^\perp,(1,-1,-1)^\perp,(3,-4,0)^\perp).
    \]
    The three other cone descriptions are proved similarly.
\end{proof}

Figure~\ref{fig:syznefcone} illustrates the movable cone with labels on the chambers we have described so far. Figure~\ref{fig:nefzoom} gives more detail further out from the central chamber.

\begin{figure}
\begin{center}
\begin{tikzpicture}[xscale=4.8, yscale=4]
        \draw[scale=1, domain=-1.183:-0.317,variable=\x,red] plot ({\x},{\x+1.5});
        \draw[scale=1, domain=0.317:1.183,variable=\x,red] plot ({\x},{-\x+1.5});
        \draw[scale=1, domain=-1.183:-.317,variable=\x,red] plot ({\x},{-\x-1.5});
        \draw[scale=1, domain=0.317:1.183,variable=\x,red] plot ({\x},{\x-1.5});
        \draw[scale=1, domain=0.0326:0.1887,variable=\x,red] plot ({\x},{(-1/11)*\x+27/22});
        \draw[scale=1, domain=0.0326:0.1887,variable=\x,red] plot ({\x},{(1/11)*\x-27/22});
        \draw[scale=1, domain=0.6275:0.7914,variable=\x,red] plot ({\x},{(-5/7)*\x+3/2});
        \draw[scale=1, domain=0.6275:0.7914,variable=\x,red] plot ({\x},{(5/7)*\x-3/2});
        \draw[scale=1, domain=0.9347:1.0518,variable=\x,red] plot ({\x},{(-7/5)*\x+21/10});
        \draw[scale=1, domain=0.9347:1.0518,variable=\x,red] plot ({\x},{(7/5)*\x-21/10});
        \draw[scale=1, domain=1.2101:1.2243,variable=\x,red] plot ({\x},{(-11)*\x+27/2});
        \draw[scale=1, domain=1.2101:1.2243,variable=\x,red] plot ({\x},{(11)*\x-27/2});
        \draw[scale=1, domain=-0.1887:-0.0326,variable=\x,red] plot ({\x},{(1/11)*\x+27/22});
        \draw[scale=1, domain=-0.1887:-0.0326,variable=\x,red] plot ({\x},{(-1/11)*\x-27/22});
        \draw[scale=1, domain=-0.7914:-0.6275,variable=\x,red] plot ({\x},{(5/7)*\x+3/2});
        \draw[scale=1, domain=-0.7914:-0.6275,variable=\x,red] plot ({\x},{(-5/7)*\x-3/2});
        \draw[scale=1, domain=-1.0518:-0.9347,variable=\x,red] plot ({\x},{(7/5)*\x+21/10});
        \draw[scale=1, domain=-1.0518:-0.9347,variable=\x,red] plot ({\x},{(-7/5)*\x-21/10});
        \draw[scale=1, domain=-1.2243:-1.2101,variable=\x,red] plot ({\x},{(-11)*\x-27/2});
        \draw[scale=1, domain=-1.2243:-1.2101,variable=\x,red] plot ({\x},{(11)*\x+27/2});
        
        \draw[scale=1, domain=-0.9628:0.9628,variable=\x,blue] plot ({\x},{0.75});
        \draw[scale=1, domain=-0.9628:0.9628,variable=\x,blue] plot ({\x},{-0.75});
        \draw[scale=1, domain=-0.9628:0.9628,variable=\y,blue] plot ({0.75}, {\y});
        \draw[scale=1, domain=-0.9628:0.9628,variable=\y,blue] plot ({-0.75}, {\y});
        \draw[scale=1, domain=-0.4841:0.4841,variable=\x,blue] plot ({\x},{1.125});
        \draw[scale=1, domain=-0.4841:0.4841,variable=\x,blue] plot ({\x},{-1.125});
        \draw[scale=1, domain=-0.4841:0.4841,variable=\y,blue] plot ({1.125}, {\y});
        \draw[scale=1, domain=-0.4841:0.4841,variable=\y,blue] plot ({-1.125}, {\y});
        \draw[scale=1, domain=-0.121:0.121,variable=\x,blue] plot ({\x},{1.2188});
        \draw[scale=1, domain=-0.121:0.121,variable=\x,blue] plot ({\x},{-1.2188});
        \draw[scale=1, domain=-0.121:0.121,variable=\y,blue] plot ({1.2188}, {\y});
        \draw[scale=1, domain=-0.121:0.121,variable=\y,blue] plot ({-1.2188}, {\y});
        \draw[scale=1, domain=1.028:1.222,variable=\x,blue] plot ({\x},{-3*\x+3.75});
        \draw[scale=1, domain=1.028:1.222,variable=\x,blue] plot ({\x},{3*\x-3.75});
        \draw[scale=1, domain=-1.222:-1.028,variable=\x,blue] plot ({\x},{-3*\x-3.75});
        \draw[scale=1, domain=-1.222:-1.028,variable=\x,blue] plot ({\x},{3*\x+3.75});
        \draw[scale=1, domain=0.094:0.665,variable=\x,blue] plot ({\x},{(1/3)*\x-1.25});
        \draw[scale=1, domain=0.094:0.665,variable=\x,blue] plot ({\x},{-(1/3)*\x+1.25});
        \draw[scale=1, domain=-0.665:-0.094,variable=\x,blue] plot ({\x},{-(1/3)*\x-1.25});
        \draw[scale=1, domain=-0.665:-0.094,variable=\x,blue] plot ({\x},{(1/3)*\x+1.25});
        
        \draw[very thin] (0,0) circle (1.2247);
		
        \node[scale=1.2] at (0, 0){$F$};
        \node[scale=1.2] at (0, 0.9375){$F_2^\vee$};
        \node[scale=1.2] at (0, -0.9375){$F_2$};
        \node[scale=1.2] at (0.9375, 0){$F_1^\vee$};
        \node[scale=1.2] at (-0.9375, 0){$F_1$};     
\end{tikzpicture}
\end{center}
\caption{A cross-section of the chambers of the movable cone of $F$, bounded by the positive cone. Prime exceptional walls are drawn in red, and the remaining walls are drawn in blue.}\label{fig:syznefcone}
\end{figure}
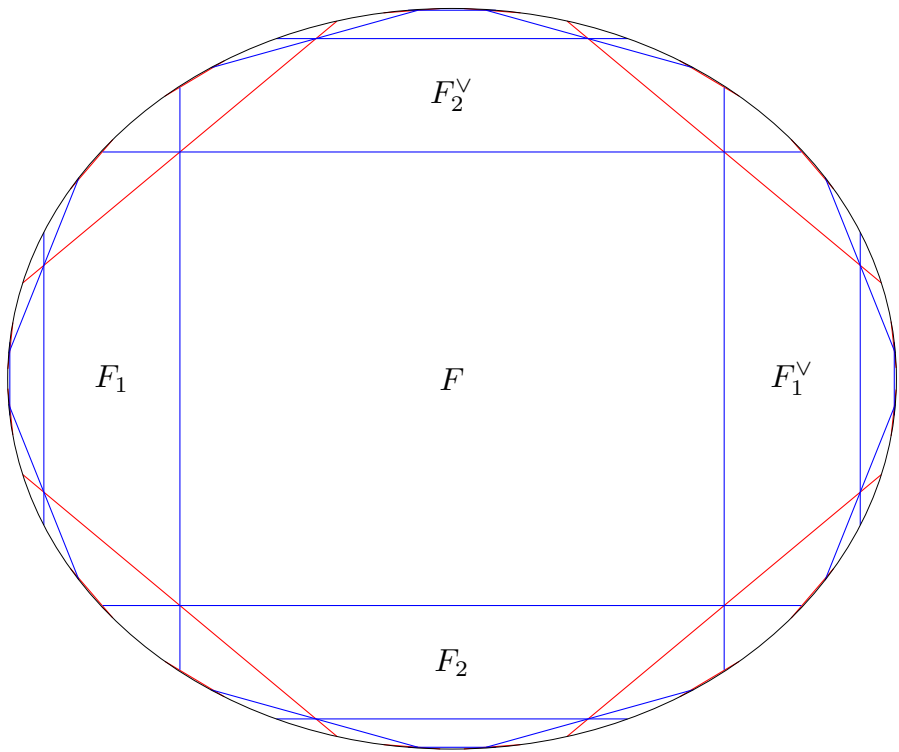

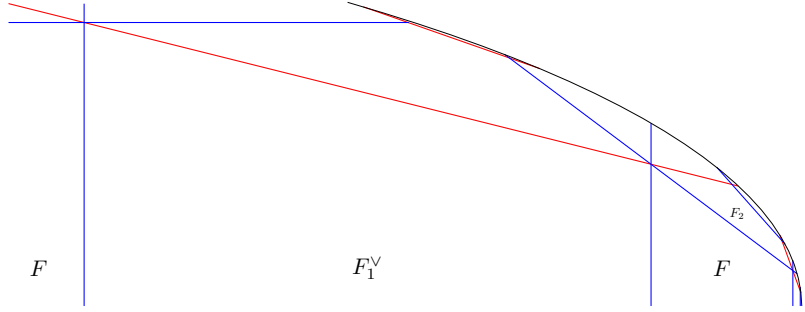
\begin{figure}
\begin{center}
\begin{tikzpicture}[xscale=20, yscale=5]
        \draw[scale=1, domain=0.7:1.183,variable=\x,red] plot ({\x},{-\x+1.5});

        \draw[scale=1, domain=1.2102:1.2243,variable=\x,red] plot ({\x},{-11*\x+(27/2)});

        \draw[scale=1, domain=0.9347:1.0518,variable=\x,red] plot ({\x},{(-7/5)*\x+21/10});

        \draw[scale=1, domain=0.7:0.965,variable=\x,blue] plot ({\x},{0.75});

        \draw[scale=1, domain=0:0.8,variable=\y,blue] plot ({0.75}, {\y});

        \draw[scale=1, domain=0:0.4841,variable=\y,blue] plot ({1.125}, {\y});

        \draw[scale=1, domain=0:0.121,variable=\y,blue] plot ({1.2188}, {\y});

        \draw[scale=1, domain=0:0.05058,variable=\y,blue] plot ({1.2237}, {\y});

        \draw[scale=1, domain=1.028:1.222,variable=\x,blue] plot ({\x},{-3*\x+3.75});
        
        \draw[scale=1, domain=1.1684:1.214,variable=\x,blue] plot ({\x},{-(36/8)*\x+(45/8)});

        \draw [very thin,domain=0:41] plot ({1.2247*cos(\x)}, {1.2247*sin(\x)});

        \node[scale=.85] at (0.72, .1){$F$};
        \node[scale=.85] at (0.9375, 0.1){$F_1^\vee$};
        \node[scale=.85] at (1.1719, .1){$F$};
        \node[scale=.5] at (1.182, 0.245){$F_2$};
        
\end{tikzpicture}
\end{center}
\caption{A detail of the figure above illustrating the chambers adjacent to $F_i^\vee$. Chambers are labeled by the isomorphism type of the model they represent. The involution $\iota_1^\vee$ exchanges the two chambers adjacent to $\Nef(F_1^\vee)$.}
\label{fig:nefzoom}
\end{figure}

We finish the section by demonstrating how to propagate all of the (infinitely many) walls of the movable cone. This is not necessary for enumerating the birational models of $F$ up to isomorphism, but we will obtain a description of the birational automorphism group of $F$ as a byproduct.

Let $\Gamma\subset\Orth(\NS(F))$ be the subgroup generated by $\iota_1^*,$ $(\iota_1^\vee)^*$, $\iota_2^*$, and $(\iota_2^\vee)^*$. Since these lattice automorphisms are induced by birational automorphisms of $F$, they preserve $\Mov(F)$ and act on the sets $\calW_{\flop}$ and $\calW_{\pex}$ of wall divisors (see Section~\ref{subsec: HK} for the definitions). In particular, they also act on the sets
\[
\
\Delta_{\flop}\coloneqq \{\rho\in\calW_{\flop}\;|\;\rho^\perp\cap\overline{\Mov(F)}\neq\varnothing\}
\]
and
\[
\
\Delta_{\pex}\coloneqq \{\rho\in\calW_{\pex}\;|\;\rho^\perp\cap\overline{\Mov(F)}\neq\varnothing\},
\]
which define walls between chambers in $\overline{\Mov(F)}$ and walls bounding $\overline{\Mov(F)}$, respectively.

\begin{lemma}\label{lemma:syzygetic-10s}
    $\Gamma$ acts freely on $\Delta_{\flop}$ with four orbits, represented by the classes $(1,\pm2,0)$ and $(1,0,\pm2)$.
\end{lemma}
\begin{proof}
Suppose $v=(a,b,c)\in \calW_{\flop}$ is a class such that $v^\perp\cap \overline{\Mov(F)}\neq\varnothing,$ i.e $v^\perp\in \Delta_{\flop}$ is a wall divisor. If $a=1$, then $v$ is one of the four vectors listed above, thus we assume $a>1$. Note that $v^\perp$ intersects one of the walls $(1,1,1)^\perp$, $(1,1,-1)^\perp$, $(1,-1,-1)^\perp$, and $(1,-1,1)^\perp$, from which we deduce $|a|<|b|$ or $|a|<|c|$. Since also $3a^2-2b^2-2c^2=-5$, we have either $|c|<|a|<|b|$ or $|b|<|a|<|c|$. We act by $\Gamma$ and find:
    \begin{align*}
        \iota_1^*v&=(5a+4b,-6a-5b,-c)&
        \iota_2^*v&=(5a+4c,-b,-6a-5c)\\
        (\iota_1^\vee)^*v&=(5a-4b,6a-5b,-c)&
        (\iota_2^\vee)^*v&=(5a-4c,-b,6a-5c).
    \end{align*}
The four classes above all define walls of the movable cone of $F$. We claim that exactly one of the classes above has a first coordinate with smaller magnitude than $|a|$, and the other three have first coordinate with larger magnitude than $|a|$. Explicitly, if $|c|<|a|<|b|$, then:
    \begin{itemize}
        \item The inequality $|c|<|a|$ implies $|5a\pm4c|>|a|$.
        \item From $|a|>1$, $|a|>|c|$, and $3a^2-2b^2-2c^2=-5$, we obtain the inequality $|3a|>|2b|$.
        \item If $\frac ab>0$, then the inequalities $a<b$ and $|3a|>|2b|$ yield $|5a-4b|<|a|$ and $|5a+4b|>|a|$.
        \item Conversely, if $\frac ab<0$, then the inequalities $a<b$ and $|3a|>|2b|$ yield $|5a+4b|<|a|$ and $|5a-4b|>|a|$.
    \end{itemize}    
The argument for the case $|b|<|a|<|c|$ is similar. In particular, there is a unique element of  $\Gamma$ taking $v$ to a $(-10)$-class whose first coordinate is $1$, obtained as a word in the four generators mentioned above by iteratively appending the generator that reduces the magnitude of the first coordinate until its value is $1$. Hence there are exactly four orbits of the action of $\Gamma$ on $\Delta_\flop$, and it is straightforward to argue from here that the action is free. 
\end{proof}

\begin{lemma}\label{lemma:syzygetic-2s}
    $\Gamma$ acts freely on $\Delta_{\pex}$ with four orbits, each containing one of the classes $(1, \pm 1, \pm 1)$ and $(1,\pm 1, \mp 1)$.
\end{lemma}
\begin{proof}
    The argument is essentially identical to the proof of Lemma \ref{lemma:syzygetic-10s}.
\end{proof}

\begin{remark}
    In the proof of Lemma~\ref{lemma:syzygetic-10s}, we did not just prove that there is a unique element of $\Gamma$ taking $v=(a,b,c)\in\Delta_\flop$ to a class with first coordinate $1$: we proved there is a unique reduced word in the generators $\iota_i^*$ and $(\iota_i^\vee)^*$ (subject to $(\iota_i^*)^2=((\iota_i^\vee)^*)^2=1$). From this, we deduce the relations on $\Gamma$. 
\end{remark}

\begin{corollary}\label{coro:gensrelations}
    There is an isomorphism
    \[
    \langle a_1,b_1,a_2,b_2\;|\;a_1^2=a_2^2=b_1^2=b_2^2=1\rangle\xrightarrow{\sim}\Gamma
    \]
    given by $a_i\mapsto\iota_i^*$ and $b_i\mapsto(\iota_i^\vee)^*$.
\end{corollary}

We will return to the group $\Gamma$ at the end of this section, proving that it is isomorphic to $\Bir(F)$ in \Cref{prop:iotasgenerate}.

\subsection{Birational geometry of $F$}\label{subsec:syzygetic models}
With our understanding of the geometry of $\Mov(F)$, we conclude by enumerating the birational hyperk\"ahler models of $F$ up to isomorphism, completing the proof of Theorem~\ref{thm:syzmain}. We also describe the birational automorphism group of $F$ completely.

\begin{lemma}\label{lemma:isogeniespreservingF}
    The only birational automorphism $\varphi$ of $F$ such that $\varphi^*\Nef(F)=\Nef(F)$ is the identity.
\end{lemma}
\begin{proof}
    We know the walls of $\Nef(F)$ from \Cref{lemma:nef}, and direct computation shows that the isometries permuting these walls and preserving the positive cone belong to the dihedral group $D_8\cong\langle R_2,R_3\rangle$. Inspecting actions on the discriminant group, Lemma~\ref{lemma:actionondisc} verifies that the only one of these isometries induced by a birational automorphism of $F$ is the identity. So, if $\varphi^*\Nef(F)=\Nef(F)$, then $\varphi$ is the identity by Remark~\ref{remark:bir embeds}.
\end{proof}

\begin{proposition}\label{proposition:syzygeticmodels}
    The five birational models $F$, $F_i$, and $F_i^\vee$ for $i=1,2$ are pairwise non-isomorphic.
\end{proposition}
\begin{proof}
    By Lemma~\ref{lemma:isogeniespreservingF}, we verified that no isometry of $\NS(F)$ induced by a birational automorphism of $F$ fixes $\Nef(F)$; from this, we see $F$ has no nontrivial regular automorphisms. This distinguishes $F$ from the four models $F_i$ and $F_i^\vee$.

    To see that $F_1$ is non-isomorphic to the other three flops of $F$, first suppose $\varphi\in\Orth(\NS(F))$ sends $\Nef(F_1)$ to $\Nef(F_2)$. Then either $\varphi$ or $\iota_2^*\circ\varphi$ acts on $\Nef(F)$ by a nontrivial automorphism. Similarly, if $\varphi$ sends $\Nef(F_1)$ to $\Nef(F_i^\vee)$, then either $\varphi$ or $(\iota_i^\vee)^*\circ\varphi$ acts on $\Nef(F)$ by a nontrivial automorphism. Again using Lemma~\ref{lemma:isogeniespreservingF}, we conclude that $\varphi$ cannot be induced by a birational automorphism of $F$. In particular, no isomorphism exists between $F_1$ and any of the other three flops of $F$.

    A symmetric argument shows that any two flops of $F$ are non-isomorphic.
\end{proof}

\begin{proposition}\label{proposition:syzygeticmodelcount}
    Up to isomorphism, $F$ has exactly five birational hyperk\"ahler models.
\end{proposition}
\begin{proof}
    By Proposition~\ref{proposition:syzygeticmodels}, $F$ has at least five birational hyperk\"ahler models, each of which can be obtained from $F$ by flopping a plane in $F$. We will prove there are no more.

    Any birational hyperk\"ahler model of $F$ can be obtained via a finite sequence of Mukai flops, shown in \cite[Theorem 1.2]{sequenceofflops2}, building on \cite[Theorem 1.1]{sequenceofflops1}. Starting from $F$, the Mukai flops are $F_i$ and $F_i^\vee$ for $i=1,2$. Using Lemmas~\ref{lemma:syzygeticiotas}, \ref{lemma:nef}, and \ref{lemma:nefFisyzygetic}, we see by Remark~\ref{remark: when bir is reg} that the two Mukai flops of $F_i$ (respectively $F_i^\vee$) are both isomorphic to $F$. Thus any sequence of two Mukai flops starting from $F$ yields a model isomorphic to $F$. Inductively, we see that any birational model of $F$ can be obtained via a single Mukai flop.
\end{proof}

Together with the content of Section~\ref{sec:epw}, identifying the flops of $F$ with pairs of dual double EPW sextics, Propositions~\ref{proposition:syzygeticmodels} and~\ref{proposition:syzygeticmodelcount} complete the proof of Theorem~\ref{thm:syzmain}.

As mentioned previously, Proposition~\ref{proposition:syzygeticplanes} explains geometrically why, after flopping any one of the four planes in $F$, the other three planes cannot also be flopped. We see this in Figure~\ref{fig:syznefcone}, where the nef cones of $F_i, F_i^\vee$ have two flopping walls and two prime exceptional walls.

\medskip

We conclude by characterizing the birational automorphism group of $F$.

\begin{proposition}\label{prop:iotasgenerate}
    The involutions $\iota_i^*$ and $(\iota_i^\vee)^*$ for $i=1,2$ generate the birational automorphism group of $F$, i.e. $\Gamma\cong\Bir(F)$. 
\end{proposition}
\begin{proof}
    Let $\varphi\in\Bir(F)$, and let $v^\perp$ be one of the walls of $\varphi^*\Nef(F)$. By Lemma~\ref{lemma:syzygetic-10s}, there is some $\psi\in\Gamma$ such that $\psi(v)$ has first coordinate $1$. We will argue that $\varphi^*=\psi^{-1}\in\Gamma$.
    
    There are five chambers of $\Mov(F)$ having a wall of the form $w^\perp$ where $w\in\calW_\flop$ has first coordinate $1$: they are $\Nef(F)$, $\Nef(F_i)$, and $\Nef(F_i^\vee)$ for $i=1,2$, so $\psi\circ\varphi^*\Nef(F)$ is one of these five. Since $\psi\circ\varphi^*$ is induced by a birational automorphism of $F$, if $\psi\circ\varphi^*\Nef(F)=\Nef(F')$, then $F\simeq F'$. Using Proposition~\ref{proposition:syzygeticmodels}, we conclude $\psi\circ\varphi^*\Nef(F)=\Nef(F)$, and Lemma~\ref{lemma:isogeniespreservingF} forces $\psi\circ\varphi^*=1$. 

    It follows that $\Gamma$ is the image of $\Bir(F)\to\Orth(\NS(F))$. This map is an embedding by Remark~\ref{remark:bir embeds}, completing the proof.
\end{proof}

Along with Corollary~\ref{coro:gensrelations}, Proposition~\ref{prop:iotasgenerate} proves  Theorem~\ref{theo:birFsyz}.

\section{A non-syzygetic pair of cubic scrolls}\label{sec:nonsyzygetic}
Let $X\subset \bP^5$ be a smooth cubic fourfold containing a non-syzygetic pair of cubic scrolls. We also assume $X$ is \emph{very general}, meaning $\rank(A(X))=3$, $\Aut(X)=0$, and $\mathrm{End_{Hdg}}(T(F))=\{\pm1\}$. As with the syzygetic case, we describe the birational geometry of the Fano variety $F$ of lines on $X$, proving the following main result:

\begin{theorem}\label{theo:nonsyzygeticmain}
    Let $F$ be the Fano variety of lines on a very general cubic fourfold $X$ containing a non-syzygetic pair of cubic scrolls. Then $F$ has eight isomorphism classes of birational hyperk\"ahler models, represented by the following: $F$ itself, six non-isomorphic Mukai flops of $F$, one for each Lagrangian plane in $F$, and a Mukai flop of $F$ along a pair of disjoint planes in $F$. Moreover, each of the six Mukai flops of $F$ can be realized as a double EPW sextic, and the flop of $F$ along a pair of disjoint planes is isomorphic to the Fano variety of lines on another smooth cubic fourfold $X'$ containing a non-syzygetic pair of cubic scrolls.
\end{theorem}

As before, we also obtain generators for $\Bir(F)$:

\begin{theorem}\label{theo:nonsyzygeticbir}
    Let $F$ be as above. There is a surjection
    \[
    \langle a,b,c,d,e,f\;|\; a^2=b^2=c^2=d^2=e^2=f^2=1\rangle\twoheadrightarrow\Bir(F)
    \] identifying each generator with a covering involution on one of the double EPW sextics obtained as a Mukai flop of $F$.
\end{theorem}

In contrast to Theorem~\ref{theo:birFsyz}, this surjection has nontrivial kernel; for further discussion, see Remark~\ref{remark:relations}.

The outline of this section is similar to that of Section~\ref{sec:syzygetic}: in Section~\ref{subsec:nonsyzplanes}, we study the arrangement of Lagrangian planes in $F$, showing that some pairs of planes intersect and others do not. We outline in Section~\ref{subsec:nonsyzlattice} properties of the lattice $\NS(F)$ and describe the actions on this lattice by the various involutions on the flops of $F$. This allows us to describe the structure of the movable cone of $F$ in Section~\ref{subsec:nonsyzwalls}, which further enables a census of the birational hyperk\"ahler models of $F$ and a description of $\Bir(F)$ in Section~\ref{subsec:nonsyzmodels}.

\subsection{Lagrangian planes in $F$}\label{subsec:nonsyzplanes}

For a general cubic $X$ containing a non-syzygetic pair of cubic scrolls $T_1$ and $T_2$, we saw in Lemma~\ref{lemma:extrascrolls} that there is a third cubic scroll $T_3\subset X$ and $[T_i]\cdot[T_j]=1$ for all $i,j$. Let $H_i$ be the hyperplane spanned by $T_i$, and $Y_i=X\cap H_i$. As in Section~\ref{subsection:threefolds}, we have
\[
F(Y_i)= P_i\cup S_i'\cup P_i^\vee\subset F,
\]
where $P_i$ and $P_i^\vee$ are Lagrangian planes. For $i\neq j$, let $\Sigma_{ij}=X\cap H_i\cap H_j$; for general $X$, all three cubic surfaces $\Sigma_{ij}$ are smooth, by Lemma~\ref{lem:Sigmasmooth}. The lines on $\Sigma_{ij}$ are parametrized by $F(Y_i)\cap F(Y_j)$.

\begin{proposition}\label{proposition:nonsyzygeticplanes}
    Let $X$ be a cubic fourfold containing a non-syzygetic pair of cubic scrolls $T_i$ and $T_j$, with $[T_i]\cdot[T_j]=1$. Then 
    \begin{itemize}
        \item $\deg(P_i\cap P_j)=\deg(P_i^\vee\cap P_j^\vee)=6$;
        \item $\deg(S_i'\cap S_j')=15$;
        \item all other pairwise intersections between a component of $F(Y_i)$ and a component of $F(Y_j)$ are empty.
    \end{itemize}
    Moreover, these intersections are all transverse when $\Sigma=X\cap H_i\cap H_j$ is smooth.
\end{proposition}
\begin{proof}
    As in the proof of Proposition~\ref{proposition:syzygeticplanes}, we use Lemmas~\ref{lemma:transverse} and~\ref{lemma:gammaintersection} to reduce the problem to studying incidence relations between the twisted cubics $\gamma_i$, $\gamma_j$, $\gamma_i^\vee$, $\gamma_j^\vee$, and the lines on $\Sigma_{ij}$. After finding that 
    \[
    \gamma_i\sim\gamma_j\sim E_0\;\;\text{and}\;\;\gamma_i^\vee\sim\gamma_j^\vee\sim 5E_0-2\sum_{k=1}^6E_k
    \]
    on $\Sigma_{ij}$, the result follows from the incidence relations for lines on $\Sigma_{ij}$. 
\end{proof}

In particular, there are six pairs of disjoint planes among the $P_i$ and $P_i^\vee$, namely, $P_i\cup P_j^\vee$ for $i\neq j$. This means that in addition to the six flops of $F$ along a plane $P_i$ or $P_i^\vee$, there are six models which can be obtained by flopping a pair of disjoint planes. Perhaps surprisingly, we will show in \Cref{lemma:Fijs} that all six models of the last type are isomorphic to one another.

\subsection{The lattice $\NS(F)$}\label{subsec:nonsyzlattice}

Again, let $\lambda_i=\alpha(T_i-\eta_X)$. Recall from Section~\ref{subsec: Fano prelims} that $\NS(F)$ with the BBF form is isomorphic to the lattice $J_{nonsyz}$ and has discriminant group $D_{\NS(F)}\cong\bZ/6\times\bZ/2\times\bZ/6$ with factors generated by $[\frac g6]$, $[\frac{\lambda_1}2]$, and $[\frac{\lambda_1}3+\frac{\lambda_2}6]$. It is easy to check that this quadratic form represents $-10$ but not $-2$.

The isometry group of $\NS(F)$ with basis $\{g,\lambda_1,\lambda_2\}$ is generated by $\pm1$ and the four transformations below. We calculated $\Aut(\NS(F))$ via \cite{Mertens} and the Magma package \texttt{AutHyp.m}.
%
\[
R_1=\begin{pmatrix*}[r] 5 & 4 & -2 \\ -6 & -5 & 2 \\ 0 & 0 & -1 \end{pmatrix*},\hspace{0.2cm}
R_2=\begin{pmatrix*}[r] 3 & 2 & -2\\ -2 & -1 & 2 \\ 2 & 2 & -1 \end{pmatrix*},
\]
\[
R_3=\begin{pmatrix*}[r] 1 & 0 & 0 \\ 0 & 0 & 1 \\ 0 & 1 & 0 \end{pmatrix*},\hspace{0.2cm}
R_4=\begin{pmatrix*}[r] 1 & 0 & 0 \\ 0 & 1 & -1 \\ 0 & 1 & 0 \end{pmatrix*}.
\]
Note that $R_1$, $R_2$, and $R_3$ are reflections whereas $R_4$ has order $6$. The transformations $R_3$ and $R_4$ generate the dihedral group $D_{12}$.

\begin{lemma}\label{lemma:nonsyzygeticadmissible}
    An isometry $\varphi\in \Orth(\NS(F))$ is induced by a birational automorphism of $F$ if and only if $\varphi$ preserves $\overline{\mathrm{Pos}(F)}$ and acts on the index two subgroup
    \[
    \left\langle\left[\frac{g}3\right]\right\rangle\times\left\langle\left[\frac{\lambda_1}2\right]\right\rangle\times\left\langle\left[\frac{\lambda_1}3+\frac{\lambda_2}6\right]\right\rangle
    \]
    of $D_{\NS(F)}$ by $\pm\id$.
\end{lemma}
\begin{proof}
    The proof is exactly the same as in Lemma~\ref{lemma:actionondiscriminant}.
\end{proof}

Let $F_i$ and $F_i^\vee$ denote the flops of $F$ along $P_i$ and $P_i^\vee$, respectively, for $i=1,2,3$. Recall from Remark~\ref{remark:iota definition} that each of these models admits a regular involution $\iota_i$ or $\iota_i^\vee$. These involutions can also be regarded as birational involutions on $F$.

\begin{lemma}\label{lemma:nonsyzygeticiotas}
    The birational involutions $\iota_i$ and $\iota_i^\vee$  act on $\NS(F)$ by $\iota_1^*=R_1$, $(\iota_1^\vee)^*=R_4^3R_1R_4^3$, $\iota_2^*=R_4^2R_1R_4^4$, $(\iota_2^\vee)^*=R_4^5R_1R_4$, $\iota_3^*=R_4^4R_1R_4^2$, and $(\iota_3^\vee)^*=R_4R_1R_4^5$.
\end{lemma}
\begin{proof}
    The proof is similar to that of Lemma~\ref{lemma:syzygeticiotas}, and we omit most details. By \cite[Proposition 6.5]{flops} and \eqref{eqn:T3def}, 
    \begin{align*}
        &\iota_1^*(1,-1,0)=(1,-1,0) &
        &(\iota_1^\vee)^*(1,1,0)=(1,1,0)\\
        &\iota_2^*(1,0,-1)=(1,0,-1) &
        &(\iota_2^\vee)^*(1,0,1)=(1,0,1)\\
        &\iota_3^*(1,1,1)=(1,1,1) &
        &(\iota_3^\vee)^*(1,-1,-1)=(1,-1,-1).
    \end{align*}
    The only matrices satisfying these conditions, squaring to the identity, preserving the quadratic form, and acting by $\pm\id$ on the index two subgroup of the discriminant group from Lemma~\ref{lemma:nonsyzygeticadmissible} are the ones proposed.
\end{proof}

\subsection{Wall and chamber structure of $\Mov(F)$}\label{subsec:nonsyzwalls}

As in the syzygetic case, we enumerate birational models of $F$ by studying the geometry of the movable cone. Unlike before, there are no prime exceptional divisors, so $\Mov(F)=\overline{\mathrm{Pos}(F)}$. 

\begin{lemma}\label{lemma:nonsyzygeticnef}
    The nef cone of $F$ is bounded by the six walls $v^\perp$ where 
    \[
    v\in\{(1,2,0),(1,2,2),(1,0,2),(1,-2,0),(1,-2,-2),(1,0,-2)\}.
    \]
     These are the six classes with first coordinate $1$ that square to $-10$.
\end{lemma}
\begin{proof}
    Similar to the proof of Lemma~\ref{lemma:nef}, we know each of these walls induces a small contraction of $F$ so lies on the boundary of $\Nef(F)$. Moreover, there are six pairs of disjoint Lagrangian planes in $F$ by Proposition~\ref{proposition:nonsyzygeticplanes}; therefore, the six edges of the cone
    \[
    \Cone((1,2,0)^\perp,(1,2,2)^\perp,(1,0,2)^\perp,(1,-2,0)^\perp,(1,-2,-2)^\perp,(1,0,-2)^\perp)
    \]
    also lie on the boundary of $\Nef(F)$. It follows that no other walls cut into $\Nef(F)$.
\end{proof}

\begin{corollary}
    The six planes in $F$ coming from components of $F(Y_i)$ where $Y_i=X\cap H_i$ for $i=1,2,3$ are the only Lagrangian planes in $F$.
\end{corollary}

Using the involutions on the six flops of $F$, we can also describe the nef cones of the flops:

\begin{lemma}\label{lemma:nefflops}
    The nef cones of flops of $F$ are
    \begin{align*}
    \Nef(F_1)&=\Cone((1,-2,0)^\perp,(1,0,2)^\perp,(3,-4,0)^\perp,(1,-2,-2)^\perp),\\
    \Nef(F_1^\vee)&=\Cone((1,2,0)^\perp,(1,0,-2)^\perp,(3,4,0)^\perp,(1,2,2)^\perp),\\
    \Nef(F_2)&=\Cone((1,0,-2)^\perp,(1,2,0)^\perp,(3,0,-4)^\perp,(1,-2,-2)^\perp),\\
    \Nef(F_2^\vee)&=\Cone((1,0,2)^\perp,(1,-2,0)^\perp,(3,0,4)^\perp,(1,2,2)^\perp),\\
    \Nef(F_3)&=\Cone((1,2,2)^\perp,(1,0,2)^\perp,(3,4,4)^\perp,(1,2,0)^\perp),\\
    \Nef(F_3^\vee)&=\Cone((1,-2,-2)^\perp,(1,-2,0)^\perp,(3,-4,-4)^\perp,(1,0,-2)^\perp).
    \end{align*}
\end{lemma}
\begin{proof}
    We give the argument for $\Nef(F_1)$, the other calculations being similar. Since $F_1$ is the flop of $F$ along $P_1$, we know $(1,-2,0)^\perp$ lies on the boundary of $\Nef(F_1)$. The involution $\iota_1$ is regular on $F_1$, so $\iota_1^*(1,-2,0)^\perp=(3,-4,0)^\perp$ also lies on the boundary of $\Nef(F_1)$ by Theorem~\ref{thm:torelli}. Moreover, the walls $(1,0,2)^\perp$ and $(1,-2,-2)^\perp$ meet each of $(1,-2,0)^\perp$ and $(3,-4,0)^\perp$, and $\iota_1^*(1,0,2)^\perp=(1,-2,-2)^\perp$, so 
    \[
    \Nef(F_1)\subset\Cone((1,-2,0)^\perp,(1,0,2)^\perp,(3,-4,0)^\perp,(1,-2,-2)^\perp).
    \]
    Since $P_1\cap P_i^\vee=\varnothing$ for $i=2,3$ by Proposition~\ref{proposition:nonsyzygeticplanes}, one can flop $F_1$ along the strict transform of $P_i^\vee$, so the lines
    \[
    (1,-2,0)^\perp\cap(1,0,2)^\perp=\Span(2,-1,1)
    \]
    and
    \[
    (1,-2,0)^\perp\cap(1,-2,-2)^\perp=\Span(2,-2,-1)
    \]
    lie on the boundary of $\Nef(F_1)$. Their images, under $\iota_1^*$,
    \[
    \Span(4,-5,-1)=(1,-2,-2)^\perp\cap(3,-4,0)^\perp
    \]
    and
    \[
    \Span(4,-4,1)=(1,0,2)^\perp\cap(3,-4,0)^\perp,
    \]
    must also lie on the boundary of $\Nef(F_1)$. It follows that 
    \[
    \Nef(F_1)=\Cone((1,-2,0)^\perp,(1,0,2)^\perp,(3,-4,0)^\perp,(1,-2,-2)^\perp),
    \]
    as claimed.
\end{proof}

For $1\le i,j\le3$, let $F_{ij}$ be the flop of $F$ along $P_i$ and $P_j^\vee$.

\begin{lemma}\label{lemma:nefdoubleflops}
    The nef cones of the $F_{ij}$ are
    \begin{align*}
    \Nef(F_{12})&=R_2\cdot\Nef(F), &
    \Nef(F_{21})&=\iota_2^*(\iota_3^\vee)^*\iota_1^*R_2\cdot\Nef(F),\\
    \Nef(F_{13})&=\iota_1^*R_2\cdot\Nef(F), &
    \Nef(F_{31})&=\iota_3^*(\iota_2^\vee)^*R_2\cdot\Nef(F),\\
    \Nef(F_{23})&=(\iota_3^\vee)^*\iota_1^*R_2\cdot\Nef(F), &
    \Nef(F_{32})&=(\iota_2^\vee)^*R_2\cdot\Nef(F).
    \end{align*}
\end{lemma}
\begin{proof}
    Since $R_2$ preserves the BBF form, it acts by an automorphism on $\Mov(F)$, sending chambers to chambers. It is straightforward to verify that the only four chambers of $\Mov(F)$ containing $(2,-1,1)$ are $\Nef(F)$, $\Nef(F_1)$, $\Nef(F_2^\vee)$, and $R_2\cdot\Nef(F)$. The edge $\Span(2,-1,1)$ of $\Nef(F)$ corresponds to simultaneously flopping $P_1$ and $P_2^\vee$, since
    \[
    (1,-2,0)^\perp\cap(1,0,2)^\perp=\Span(2,-1,1).
    \]
    Hence $\Nef(F_{12})$ contains $(2,-1,1)$, from which we deduce $\Nef(F_{12})=R_2\cdot\Nef(F)$. The other five verifications are similar.
\end{proof}

\begin{corollary}\label{coro:sixwalls}
    The model $F_{12}$ contains exactly six Lagrangian planes, six pairs of which are disjoint.
\end{corollary}
\begin{proof}
    By Lemmas~\ref{lemma:nefdoubleflops} and~\ref{lemma:nonsyzygeticnef}, the nef cone of $F_{12}$ is a cone over a hexagon, the same as $\Nef(F)$. The six faces of this chamber correspond to Lagrangian planes in $F_{12}$; the six edges correspond to pairs of disjoint Lagrangian planes in $F_{12}$.
\end{proof}

We illustrate the movable cone as described so far in Figure~\ref{fig:nonsyznefcone}. In the next section, we will need to know the chambers bordering $\Nef(F_{12})$, enumerated by the following lemma and illustrated analogously for $\Nef(F_{32})$ in Figure~\ref{fig:nonsyznefconezoom}.

\begin{lemma}\label{lemma:adjacenttoFij}
    Let $\{i,j,k\}=\{1,2,3\}$. Then the chambers sharing a face with $\Nef(F_{ij})$ are $\Nef(F_i)$, $\Nef(F_j^\vee)$, $\iota_i^*\Nef(F_k^\vee)$, $(\iota_j^\vee)^*\Nef(F_k)$, $(\iota_j^\vee)^*\iota_k^*\Nef(F_i^\vee)$, and $\iota_i^*(\iota_k^\vee)^*\Nef(F_j)$.
\end{lemma}
\begin{proof}
    This is verified by direct calculation: using Lemmas~\ref{lemma:nefdoubleflops} and~\ref{lemma:nefflops}, one can identify the walls of $\Nef(F_{ij})$ and $\Nef(F_i)$ for all $i,j,k$. Using Lemma~\ref{lemma:nonsyzygeticiotas}, one can check that each of the chambers above shares a face with $\Nef(F_{ij})$. By Corollary~\ref{coro:sixwalls}, these six faces cover the entire boundary of $\Nef(F_{ij})$.
\end{proof}

\begin{figure}
\begin{center}
\includegraphics[scale=0.7]{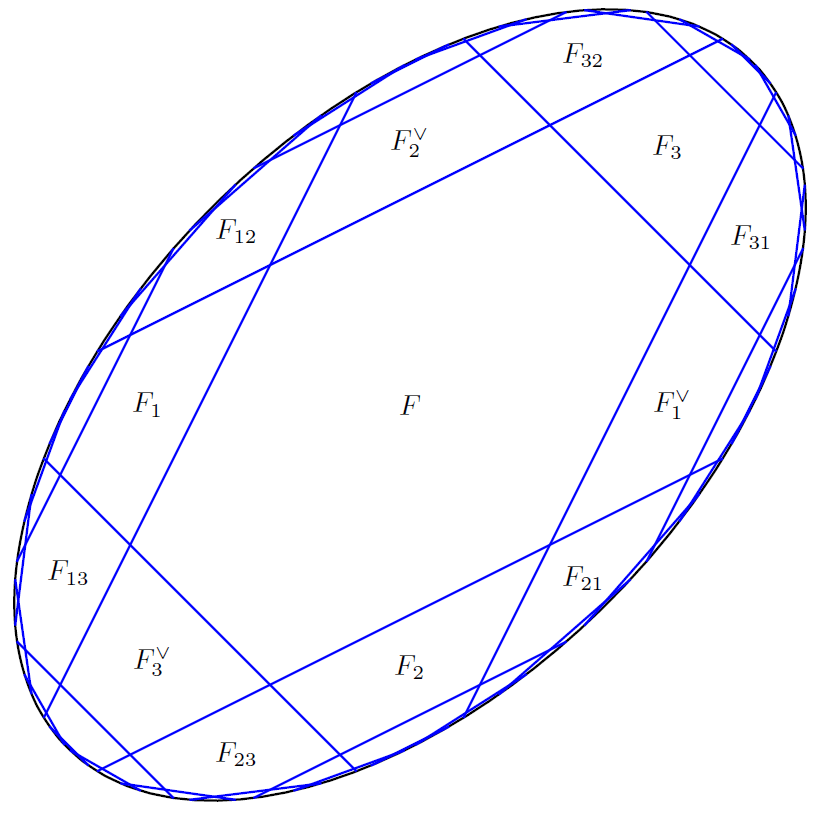}
\end{center}
\caption{A slice of the chambers of the movable cone of $F$, bounded by the positive cone.}\label{fig:nonsyznefcone}
\end{figure}

\begin{figure}
\begin{center}
\includegraphics[scale=0.58]{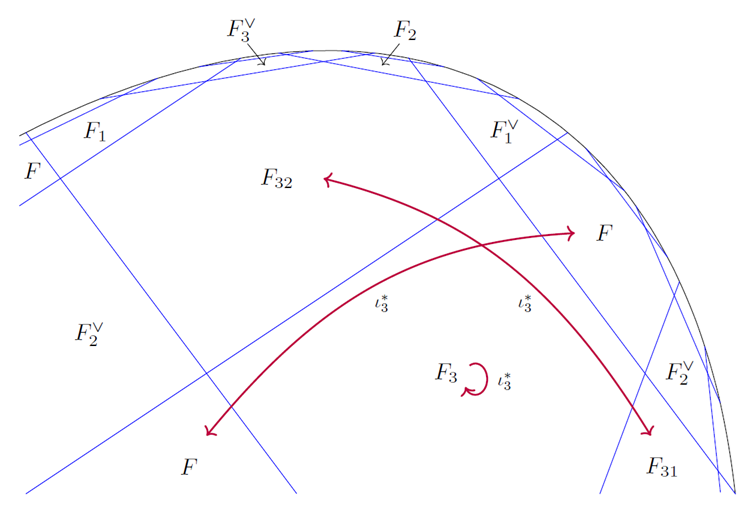}
\end{center}
\caption{A detail of the figure above illustrating the chambers adjacent to $F_3$ and $F_{32}$. Chambers are labeled by the isomorphism type of the model they represent. The arrows indicate the action of $\iota_3^*$ on the chambers.}\label{fig:nonsyznefconezoom}
\end{figure}

\subsection{Birational geometry of $F$}\label{subsec:nonsyzmodels}

We now classify birational hyperk\"ahler models of $F$ up to isomorphism, proving Theorem~\ref{theo:nonsyzygeticmain}.

\begin{lemma}\label{lemma:Fijs}
    The six birational models $F_{ij}$ for $1\le i,j\le3$ are all isomorphic.
\end{lemma}
\begin{proof}
    Since, by Theorem~\ref{thm:torelli}, $\iota_i^*$ and $(\iota_i^\vee)^*$ act on $\Mov(F)$ by exchanging chambers corresponding to the same isomorphism type, this follows from Lemma~\ref{lemma:nefdoubleflops}.
\end{proof}

\begin{proposition}\label{prop:nonsyzygeticmodels}
    The eight birational models $F$, $F_i$, $F_i^\vee$, and $F_{12}$ are pairwise non-isomorphic.
\end{proposition}
\begin{proof}
    The idea of the proof is to show that no isometry of $\NS(F)$ taking a chamber of one of the models above to a chamber of another satisfies the conditions of Lemma~\ref{lemma:nonsyzygeticadmissible}. We begin with $F_1$. Since the chamber $\Nef(F_1)$ has four walls, whereas $\Nef(F)$ and $\Nef(F_{12})$ each have six walls, we know $F_1\not\simeq F$ and $F_1\not\simeq F_{12}$.

    Let $\varphi$ be an isometry sending $\Nef(F_1)$ to $\Nef(F_i)$ for $i\neq1$ or $\Nef(F_i^\vee)$ for $i=1,2,3$. The four edges of $\Nef(F_1)$, corresponding to simultaneous flops of disjoint planes, are rays spanned by the vectors
    \[
    v_1=(2,-1,1),\,v_2=(2,-2,-1),\,v_3=(4,-4,1),\text{ and }v_4=(4,-5,-1);
    \]
    let $w_i=\varphi(v_i)$, so the vectors $w_i$ span the edges of the chamber $\varphi(\Nef(F_1))$. Then
    \[
    \varphi(0,1,2)=\varphi(v_1-v_2)=w_1-w_2.
    \]
    For $\varphi$ to be induced by a birational automorphism of $F$, we must have $\varphi([\lambda_1/2])=[\lambda_1/2]$ in the discriminant group, by Lemma~\ref{lemma:nonsyzygeticadmissible}. Hence $\varphi(0,1,2)=(a,b,c)$ where $a$ and $c$ are even and $b$ is odd. By direct inspection, the only chamber other than $\Nef(F_1)$ having walls $w_1$, $w_2$ such that $w_1-w_2=(a,b,c)$ with $a$ and $c$ even and $b$ odd is $\Nef(F_1^\vee)$. Hence $F_1$ is not isomorphic to $F_i$ or $F_i^\vee$ for $i=2,3$.

    We now verify $F_1\not\simeq F_1^\vee$. There are eight possibilities for $w_1$ and $w_2$:
    \begin{multicols}{2}
    \begin{enumerate}
        \item[(a)] $w_1=(2,2,1)$, $w_2=(2,1,-1)$ 
        \item[(b)] $w_1=(2,1,-1)$, $w_2=(2,2,1)$ 
        \item[(c)] $w_1=(2,1,-1)$, $w_2=(4,4,-1)$ 
        \item[(d)] $w_1=(2,2,1)$, $w_2=(4,5,1)$ 
        \item[(e)] $w_1=(4,4,-1)$, $w_2=(4,5,1)$ 
        \item[(f)]$w_1=(4,5,1)$, $w_2=(4,4,-1)$ 
        \item[(g)] $w_1=(4,5,1)$, $w_2=(2,2,1)$ 
        \item[(h)] $w_1=(4,4,-1)$, $w_2=(2,1,-1)$. 
    \end{enumerate}
    \end{multicols}
    Since $\iota_1^\vee$ is regular on $F_1^\vee$, and $(\iota_1^\vee)^*(2,2,1)=(4,4,-1)$ and $(\iota_1^\vee)^*(2,1,-1)=(4,5,1)$, composing $\varphi$ with $(\iota_1^\vee)^*$ reduces cases (e), (f), (g), and (h) to (a), (b), (c), and (d), respectively. In case (a), we find $\varphi=R_4R_3$; in (b), we find $\varphi=R_4^3$; in case (c), we find $\varphi=R_4^3R_2$; in case (d), we find $\varphi=R_4R_3R_2$. We again check the action of $\varphi$ on the discriminant group, concluding by  Lemma~\ref{lemma:nonsyzygeticadmissible} that $\varphi$ cannot be induced by a birational automorphism of $F$.

    By symmetry, all six models $F_i$ and $F_i^\vee$ are pairwise non-isomorphic. It remains to check that $F\not\simeq F_{12}$. Suppose for the sake of contradiction that $\varphi$ is an isometry of $\NS(F)$ induced by an isomorphism $F_{12}\xrightarrow{\sim} F$, so $\varphi(\Nef(F))=\Nef(F_{12})$. Then $\varphi$ maps the six chambers $F_i$, $F_i^\vee$ for $i=1,2,3$ to the six chambers sharing faces with $\Nef(F_{12})$. In light of Lemma~\ref{lemma:adjacenttoFij} and the above, we deduce $\varphi(\Nef(F_1))=\Nef(F_1)$ and $\varphi(\Nef(F_2^\vee))=F_2^\vee$.  Direct computation verifies that the only isometry with $\varphi(\Nef(F))=F_{12}$ satisfying these conditions is $R_2$, but by Lemma~\ref{lemma:nonsyzygeticadmissible}, $R_2$ is not induced by a birational automorphism of $F$, yielding a contradiction.
\end{proof}

\begin{proposition}\label{prop:nonsyzygeticmodelcount}
    Up to isomorphism, $F$ has eight birational hyperk\"ahler models.
\end{proposition}
\begin{proof}
    By Proposition~\ref{prop:nonsyzygeticmodels}, $F$ has at least eight birational hyperk\"ahler models, each of which can be obtained from $F$ by flopping a single plane or a pair of disjoint planes. We will prove there are no more.

    As in Proposition~\ref{proposition:syzygeticmodels}, any birational hyperk\"ahler model of $F$ can be obtained via a finite sequence of Mukai flops. Starting from $F$, the Mukai flops are $F_i$ and $F_i^\vee$ for $i=1,2,3$. Using Lemmas~\ref{lemma:nefflops} and \ref{lemma:nefdoubleflops}, the four Mukai flops of $F_i$ (respectively $F_i^\vee$) are isomorphic to $F$, $F_{ij}$, and $F_{ik}$ (respectively $F$, $F_{ji}$, and $F_{ki}$). Applying Lemma~\ref{lemma:Fijs}, we see that a sequence of two Mukai flops starting at $F$ yields a model isomorphic either to $F$ or to $F_{12}$. By Lemma~\ref{lemma:adjacenttoFij}, a third flop yields a model isomorphic to $F_i$ or $F_i^\vee$ for some $i=1,2,3$. Inductively, we conclude that an odd number of Mukai flops starting from $F$ yields a model isomorphic to $F_i$ or $F_i^\vee$, and an even number of Mukai flops yields a model isomorphic to $F$ or $F_{12}$; in particular, any birational hyperk\"ahler model of $F$ is isomorphic to one of these.
\end{proof}

The content of Section~\ref{sec:epw} identifies each flop of $F$ with a double EPW sextic. To complete the proof of Theorem~\ref{theo:nonsyzygeticmain}, we identify $F_{12}$ with the Fano variety of lines on another smooth cubic fourfold:

\begin{proposition}
    There is a unique smooth cubic fourfold $X'$ containing a non-syzygetic pair of cubic scrolls whose Fano variety of lines is isomorphic to $F_{12}$. Moreover, $X\not\simeq X'$.
\end{proposition}
\begin{proof}
    Lemma~\ref{lemma:nefdoubleflops} states $\Nef(F_{12})=R_2\cdot\Nef(F)$, so $F_{12}$ contains an ample class of square $6$ and divisibility $2$, namely $g'=R_2\cdot g$. Hence the pair $(F_{12},g')$ defines a point in the moduli space $\calM_6^{(2)}$ of hyperk\"ahler fourfolds of K3$^{[2]}$-type with a polarization of square $6$ and divisibility $2$. Let $\calM_{cub}$ be the moduli space of marked cubic fourfolds, so there is a commutative diagram of period maps
    \[\xymatrix{ \calM_{cub} \ar[rr]^{F} \ar[dr]_{p'} & & \calM_6^{(2)} \ar[dl]^{p} \\ 
            & \calP_6^{(2)} &
    }.\]
    By \cite{laz09} and \cite{Loo}, the complement of the image of $p'$ is the union of the Heegner divisors $\calD_{6,2}^{(2)}$ and $\calD_{6,6}^{(2)}$; for more discussion, see \cite[Appendix B]{debarre2020hyperkahler}. Hence to show that $F_{12}$ is the Fano variety of lines on a smooth cubic fourfold, it suffices to check that there are no classes $v\in g'^\perp\cap \NS(F_{12})$ with $\mathrm{div}(v)=2$ and $v^2\in\{-2,-6\}$. Indeed, if such a vector existed, then $R_2\cdot v$ would have the same numerics and lie in $g^\perp\cap \NS(F)$; recalling the Abel--Jacobi map described in Section~\ref{subsec: Fano prelims}, this would force $X$ to be of discriminant $2$ or $6$. But $X$ is smooth, so $X\not\in\calC_2\cup\calC_6$, and no such class $v$ exists.
    
    We have shown $F_{12}$ is the Fano variety of a smooth cubic fourfold $X'$. Moreover, $X'$ is unique since $p'$ is injective \cite[Proposition 6]{charles2012remark}. The Abel--Jacobi map allows us to deduce the intersection form on $A(X')$, proving that $X'$ contains a non-syzygetic pair of cubic scrolls. Finally, the fact from Proposition~\ref{prop:nonsyzygeticmodels} that $F\not\simeq F_{12}$ forces $X\not\simeq X'$.
\end{proof}

We end by providing some information about the structure of $\Bir(F)$. Let 
    \[
    \Gamma=\langle\iota_i^*,(\iota_i^\vee)^*\;|\;i=1,2,3\rangle\subset\Orth(\NS(F)).
    \] 
    Unlike in the syzygetic case, we are not able to make use of the action of $\Gamma$ on $\Delta_{\mathrm{flop}}$ in order to deduce generators and relations for $\Bir(F)$; largely, this is because in the syzygetic case, each wall of $\Mov(F)$ was adjacent to finitely many chambers whereas in the non-syzygetic case, each wall borders infinitely many chambers. On the other hand, the fact that the movable cone coincides with the positive cone somewhat streamlines the argument in the following lemma, analogous to Lemma~\ref{lemma:syzygetic-10s}.

\begin{lemma}\label{lemma:square6}
    The group $\Gamma$
    acts on the set $\{v\in\NS(F)\,|\,v^2=6\}$ with at most seven orbits, represented by the classes $(1,0,0)$, $(3,\pm2,\pm4)$, $(3,\pm4\pm2)$, and $(3,\pm2,\mp2)$.
\end{lemma}
\begin{proof}
    As in Lemma~\ref{lemma:syzygetic-10s}, one starts with an arbitrary class $v=(a,b,c)$ such that $v^2=6$ and $a>3$ and finds that at least one of the involutions $(\iota_i)^*$ or $(\iota_i^\vee)^*$ for $i=1,2,3$ reduces the magnitude of the first coordinate. Iterating this process, one obtains one of the seven classes $v\in\NS(F)$ with $v^2=6$ and first coordinate no larger than $3$.
\end{proof}

\begin{remark}\label{remark:relations}
    Unlike in the syzygetic case, there is not a unique sequence of $(\iota_i)^*$ and $(\iota_i^\vee)^*$ taking an arbitrary class $v\in\NS(F)$ with $v^2=6$ to one with first coordinate at most $3$: indeed, for $\{i,j,k\}=\{1,2,3\}$, we note $\iota_i\circ\iota_j^\vee\circ\iota_k=\iota_k^\vee\circ\iota_j\circ\iota_i^\vee$. In particular, the action of Lemma~\ref{lemma:square6} does not afford a characterization of the relations among the generators of $\Gamma$. Nevertheless, we obtain generators for $\Bir(F)$.
\end{remark}

\begin{proposition}
    The birational involutions $\iota_i^*$ and $(\iota_i^\vee)^*$ for $i=1,2,3$ generate the birational automorphism group of $F$, i.e. $\Gamma\cong\Bir(F)$.
\end{proposition}
\begin{proof}
    Let $\varphi\in\Bir(F)$, and let $\varphi^*\Nef(F)=\Nef(F')$. Then $\varphi^*(g)=v$ for some class $v$ with $q(v)=6$. By Lemma~\ref{lemma:square6}, there is some $f\in\Gamma$ such that $f(v)$ has first coordinate at most $3$. Moreover, since the generators of $\Gamma$ are induced by birational automorphisms, we have $f=\psi^*$ for some $\psi\in\Bir(F)$. Now, 
    \[
    (\varphi\circ\psi)^*(g)\in\{(1,0,0),(3,\pm2,\pm4),(3,\pm4,\pm2),(3,\pm2,\mp2)\},
    \]
    but on the other hand, $(\varphi\circ\psi)^*\Nef(F)=\Nef(F'')$ where $F''\simeq F$. Using Lemma~\ref{lemma:nefdoubleflops}, the classes $(3,\pm2,\pm4)$, $(3,\pm4\pm2)$, and $(3,\pm2,\mp2)$ belong to the nef cones of $F_{ij}$ for $1\le i\neq j\le 3$, so by Proposition~\ref{prop:nonsyzygeticmodels} we obtain $(\varphi\circ\psi)^*(g)=g$. The subgroup of $\Orth(\NS(F))$ of isometries fixing $g$ is the dihedral group generated by $R_3$ and $R_4$, and using Lemma~\ref{lemma:nonsyzygeticadmissible}, the only one of these isometries induced by a birational automorphism of $F$ is the identity. By Remark~\ref{remark:bir embeds}, the map $\Bir(F)\to\Orth(\NS(F))$ is an embedding, so $\varphi\circ\psi=\id$, and $\varphi=\psi^{-1}\in\Gamma$.
\end{proof}

This completes the proof of Theorem~\ref{theo:nonsyzygeticbir}.

\appendix

\section{Examples}\label{appendix:examples}
Here, we provide explicit examples of cubic fourfolds containing pairs of cubic scrolls in order to justify earlier assertions about generic behavior (cf.~Lemma~\ref{lem:Sigmasmooth}). The computational claims in the proof below can be verified with Magma code provided on the arXiv as an ancillary file.

We work over the field $\bF_{29}$, but the choices made in producing our examples amount to picking a point in a tower of projective bundles, as can be seen in the Magma code. Hence our examples lift to characteristic zero.

\subsection{An explicit cubic fourfold with a syzygetic pair of cubic scrolls}\label{ex: smth inters}

Let $X$ be the smooth cubic fourfold with defining equation
\begin{align*}
f= 17x_0x_1x_2 &+ 19x_1^2x_2 + 9x_0x_2^2 + 10x_1x_2^2 + 18x_2^3 + 12x_0^2x_3 + 10x_0x_1x_3 + 8x_0x_2x_3 \\
&+ 4x_1x_2x_3 + 27x_2^2x_3 +
    2x_0x_3^2 + 3x_2x_3^2 + 20x_0^2x_4 + 11x_0x_1x_4 \\
    &+ 23x_1^2x_4 + 11x_0x_2x_4 + 24x_1x_2x_4 + 14x_2^2x_4 + 7x_0x_3x_4 + 
    26x_1x_3x_4 \\
    &+ 19x_2x_3x_4 + 15x_0x_4^2 + 10x_1x_4^2 + 7x_0^2x_5 + 16x_0x_1x_5 + 18x_1^2x_5\\
    &+ 22x_0x_3x_5 + 8x_1x_3x_5 + 
    23x_3^2x_5 + 18x_0x_4x_5 + 5x_1x_4x_5 + 7x_3x_4x_5\\
    &+ 22x_4^2x_5 + 21x_1x_5^2 + 5x_3x_5^2 + 28x_4x_5^2 + 2x_5^3.
\end{align*}
The hyperplanes $H_1$ and $H_2$ defined by $x_5=0$ and $x_2=0$, respectively, intersect $X$ in six-nodal cubic threefolds $Y_1$ and $Y_2$. The cubic scroll $T_1\subset H_1$, defined by the vanishing of the minors of the matrix
\[
M_1=\left(\begin{matrix}
    x_0& x_1& x_2\\
x_2 &x_3 &x_4
\end{matrix}\right),
\]
is contained in $Y_1$. Similarly, the cubic scroll $T_2\subset H_2$, defined by the vanishing of the minors of the matrix
\[
M_2=\left(\begin{matrix}
    l_0 &  l_1 &  l_2\\
l_2 &  l_3 &  l_4
\end{matrix}\right),
    \]
where
\begin{align*}
    l_0&\coloneqq 17x_0 + 12x_1 + 12x_2 + 17x_3 + 7x_4 + 6x_5,\\
    l_1&\coloneqq 17x_0 + 4x_1 + 17x_2 + 25x_3 + 18x_4 + 13x_5,\\
    l_2&\coloneqq x_5,\\
    l_3&\coloneqq 10x_0 + 13x_1 + 12x_2 + 15x_3 + 14x_4 + 17x_5, \\
    l_4&\coloneqq 16x_0 + 13x_1 + 9x_2 + 10x_3 + 19x_4 + 7x_5, \\
\end{align*}

\noindent is contained in $Y_2$. The pair of cubic scrolls is syzygetic; $T_1$ and $T_2$ intersect transversely in three points, namely $(1:1:0:0:0:0)$, $(0:0:0:1:1:0)$, and $(0:1:0:1:0:0)$. The cubic surface $\Sigma=Y_1\cap Y_2$ is smooth, as desired.

\subsection{An explicit cubic fourfold with a non-syzygetic pair of cubic scrolls}

Let $X$ be the smooth cubic fourfold with defining equation
\begin{align*}
f = 20x_0x_1x_2 &+ 4x_1^2x_2 + 15x_0x_2^2 + 17x_1x_2^2 + 6x_2^3 + 9x_0^2x_3 + 25x_0x_1x_3 + 27x_0x_2x_3 \\
&+ 15x_1x_2x_3 + 19x_2^2x_3 + 5x_0x_3^2 + 19x_2x_3^2 + 14x_0^2x_4 + 14x_0x_1x_4 \\
&+ 9x_1^2x_4 + 23x_0x_2x_4 + 25x_1x_2x_4 + 21x_2^2x_4
    + 14x_0x_3x_4 + 10x_1x_3x_4 \\
&+ 18x_2x_3x_4 + 8x_0x_4^2 + 11x_1x_4^2 + 10x_0^2x_5 + 4x_0x_1x_5 + 24x_1^2x_5 \\
&+ 22x_0x_3x_5 + 16x_1x_3x_5 + 17x_3^2x_5 + 5x_0x_4x_5 + 18x_1x_4x_5 + 25x_3x_4x_5 \\
&+ 27x_4^2x_5 + 2x_0x_5^2 + 28x_1x_5^2 + 21x_3x_5^2 + 28x_4x_5^2 + 13x_5^3.
\end{align*}
We consider the following hyperplanes:
\begin{align*}
H_1&\coloneqq \{x_5=0\},\\
H_2&\coloneqq \{x_2=0\}, \\
H_3&\coloneqq \{x_0 + 24x_1 + x_2 + x_3 + 20x_4 + 9x_5=0\}.
\end{align*}
One sees that $H_1, H_2$ and $H_3$ intersect $X$ in six-nodal cubic threefolds $Y_1$, $Y_2$, and $Y_3$ respectively. The cubic scroll $T_1\subset H_1$ defined by the vanishing of the minors of the matrix
\[
M_1=\left(\begin{matrix}
    x_0& x_1& x_2\\
x_2 &x_3 &x_4
\end{matrix}\right),
\]
is contained in $Y_1$.  The cubic scroll $T_2 \subset H_2$, defined by the vanishing of the minors of the matrix
\[
M_2=\left(\begin{matrix}
    l_0 &  l_1 &  l_2\\
l_2 &  l_3 &  l_4
\end{matrix}\right),
    \]
where
\begin{align*}
    l_0&\coloneqq 26x_0 + 2x_1 + 8x_2 + 4x_3 + 5x_4 + 24x_5,\\
    l_1&\coloneqq 13x_0 + 11x_1 + 22x_2 + 18x_3 + 6x_4 + 15x_5,\\
    l_2&\coloneqq 12x_0 + 19x_1 + 15x_2 + 16x_3 + 17x_4 + 14x_5,\\
    l_3&\coloneqq 18x_0 + 3x_1 + 18x_2 + 26x_3 + 18x_4 + 10x_5, \\
    l_4&\coloneqq 28x_0 + 14x_1 + 5x_2 + 21x_3 + x_4 + 3x_5. \\
\end{align*}
is contained in $Y_2$. The cubic scroll $T_3\subset H_3$,  cut out by the quadrics
\begin{align*}
    Q_{31}&\coloneqq x_1^2 + 17x_1x_3 + 27x_2x_3 + 9x_3^2 + 27x_1x_4 + 23x_2x_4 + 11x_3x_4 + 14x_4^2 + 24x_1x_5  \\
    &\hspace{2cm} + 13x_2x_5 + 10x_3x_5 + 2x_4x_5 + 8x_5^2,\\
Q_{32}&\coloneqq x_1x_2 + 25x_1x_3 + 20x_2x_3 + 5x_3^2 + 5x_1x_4 + 6x_2x_4 + 23x_3x_4 + 5x_4^2 + 20x_1x_5  \\
&\hspace{2cm} + 2x_2x_5 + 24x_3x_5 + 3x_4x_5 + 8x_5^2,\\
Q_{33}&\coloneqq x_2^2 + 28x_1x_3 + 5x_2x_3 + 25x_3^2 + 20x_1x_4 + 14x_2x_4 + 15x_3x_4 + x_4^2 + 27x_1x_5  \\
&\hspace{2cm}+ 15x_2x_5 + 10x_3x_5 + 2x_4x_5 + x_5^2,
\end{align*}
is contained in $Y_3$. The pairs $(T_1,T_2)$, $(T_1,T_3)$, and $(T_2,T_3)$ all form non-syzygetic pairs. Indeed, any two intersect transversely in one point; explicitly, 
\begin{align*}
    T_1\cap T_2&=\{(0:1:0:1:0:0)\},\\
    T_1\cap T_3&=\{(22 : 19 : 15 : 9 : 1 : 0)\},\\
    T_2\cap T_3&=\{(15 : 9 : 0 : 15 : 9 : 1)\}.
\end{align*}

To verify that $[T_3]=3\eta_X-[T_1]-[T_2]$, it suffices to check $[T_3]\in\langle\eta_X,[T_1],[T_2]\rangle$. If not, then each of the non-syzygetic pairs gives rise to another $\bF_{29}$-rational hyperplane slicing $X$ in a cubic threefold singular along at least a length $6$ zero-dimensional subscheme. In that case, $X$ has at least six such hyperplane sections. Direct computation verifies that $X$ has only four six-nodal hyperplane sections, cut out by $H_1$, $H_2$, $H_3$, and
\[
x_0 + 16x_1 + 8x_2 + 11x_3 + 26x_4 + 13x_5=0.
\]
So, $T_3$ represents the desired class in cohomology.

Now, let $\Sigma_{ij}=Y_i\cap Y_j$. All three cubic surfaces $\Sigma_{ij}$ are smooth, so the same is true for a general cubic fourfold containing a non-syzygetic pair of cubic scrolls.

\bibliographystyle{alpha}
\bibliography{bibliography}

\end{document}